\documentclass[a4paper,12pt]{amsart}
\usepackage[cp1250]{inputenc}
\usepackage{amssymb}
\usepackage{amscd}
\usepackage{graphicx}
\usepackage{epstopdf}
\newtheorem{thm}{\sc Theorem}
\newtheorem{lemma}[thm]{\sc Lemma}
\newtheorem{prop}[thm]{\sc Proposition}
\newtheorem{cor}[thm]{\sc Corollary}
\theoremstyle{definition} \newtheorem{defi}{\sc Definition}
\newtheorem{ex}[thm]{\sc Example}
\newtheorem{rem}[thm]{\sc Remark}

\DeclareMathOperator{\ii}{i}

\newcommand{\<}{\langle}
\renewcommand{\>}{\rangle}
\newcommand{\seq}[1]{\left\langle #1 \right\rangle}
\newcommand{\set}[1]{\left\{ #1 \right\}}

\DeclareMathOperator{\dom}{dom}
\DeclareMathOperator{\rank}{rank}
\DeclareMathOperator{\RE}{Re}

\newcommand{\A}{\mathbf{A}}
\newcommand{\B}{\mathcal{B}}
\newcommand{\C}{\mathbb{C}}
\newcommand{\Comp}{\C}
\newcommand{\D}[2]{{#2}\!\otimes\!{#1}}

\renewcommand{\H}{\mathcal{H}}

\newcommand{\N}{\mathbb{N}}
\newcommand{\norm}[1]{\left\|#1\right\|}

\newcommand{\vp}{\varphi}
\newcommand{\R}{\mathbb{R}}
\newcommand{\Real}{\mathbb{R}}
\newcommand{\U}{$\mathbf{U}$} 

\newcommand{\WT}{$\mathbf{W}$} 

\usepackage{color}

   \def\uwat#1{\textcolor{red}{$\blacktriangleright$}{\sl #1}
   \textcolor{red}{$\blacktriangleleft$}}

\title[Rank two perturbations]{Rank two perturbations of matrices and operators and operator model for $\mathbf{t}$-transformation of probability measures}

\thanks{Corresponding author's email: janusz.wysoczanski@math.uni.wroc.pl, tel.: +48 71 3757095. 
}

\author{Anna Kula}
\address{AK: Institute of Mathematics, University of Wrocław, pl. Grunwaldzki 2/4, 50-384 Wrocław, Poland}
\email{anna.kula@math.uni.wroc.pl}
\author{Micha{\l} Wojtylak}
\address{MW: Institute of Mathematics, Jagiellonian University, ul. Łojasiewicza 6, 30-348 Krak{\'o}w, Poland} 
\email{michal.wojtylak@im.uj.edu.pl}
\thanks{MW acknowledges the financial support by the NCN (National Science Center) grant, decision No. DEC-2013/11/B/ST1/03613}
\author{Janusz Wysocza\'nski} 
\address{JW: Institute of Mathematics, University of Wrocław, pl. Grunwaldzki 2/4, 50-384 Wrocław, Poland}
\email{janusz.wysoczanski@math.uni.wroc.pl}

\begin{document}

\maketitle

\hyphenation{va-ni-shing}

 \begin{abstract}
Rank two parametric perturbations of operators and matrices are studied in various settings. 
In the finite dimensional case the formula for a characteristic polynomial is derived and the large parameter asymptotics of the spectrum is computed. The large parameter asymptotics of a rank one perturbation of singular values and condition number are discussed as well. 
In the operator case the formula for a rank two transformation of the spectral measure is derived and it appears to be the $t$-transformation of a probability measure,  studied previously in the free probability context. New transformation of measures is studied and several examples are presented. 
 \end{abstract}

\section*{Introduction}

The aim of this paper is to investigate the rank two deformations of operators. We recall that several papers have studied the topic of rank one perturbations of matrices (e.g. \cite{MBO,MMRR1,MMRR2,RW12,Sa0,Sa1,Sa2}) and operators (e.g. \cite{DHS1,DHS3,SWW1,SWW2,VL}), also rank $k$ perturbations of matrices were recently considered in \cite{Batzkeetal,Thompson76}.  Nonetheless, rank two perturbations seems to be a topic that has not attracted much attention, despite its role in mathematical modelling problems. 
Let us mention some of the situations, where rank two perturbations appear naturally.

When a physical system is modelled by a linear ODE with constant coefficients $x'=Ax$ frequently the physical laws of the system impose a particular structure of the entries of the matrix.   e.g. the canonical equations of the classical mechanics lead to a Hamiltonian matrix  $A$, i.e. a matrix of the form
$$
A=\begin{pmatrix} A_0 & A_1\\ A_2 & -A_0^\top\end{pmatrix},\quad A_1=A_1^\top,\ A_2=A_2^\top. 
$$
Observe that every rank one perturbation of $A_0$ results in a rank two perturbation of $A$ and a change in one off diagonal entry of $A_1$ or $A_2$  implies  change in the symmetric entry  and gives a rank two deformation of $A$. The topic of stability of rank two perturbations of Hamiltonian systems 
is discussed further in Remark \ref{Phasetr}.  Similar problems occur in modelling electronic circuits \cite{Freund10}, where a change in one parameter of the electric network (e.g. cutting the electric transmitter, increasing the capacity of one capacitor, etc.) leads to a rank two perturbation of a linear pencil.
 Another application of the theory of rank two perturbations comes from  modelling the polydisperse sedimentation, see e.g. \cite{Berres1,Berres2}. The necessary condition for a successful modelling is that the perturbed matrix has simple, real and distinct eigenvalues.  Based on our calculations we will provide a criterion on the spectrum of a rank two perturbations of a matrix being real and simple, essentially simpler than the existing necessary and sufficient condition in \cite{DM}.
 
%

Furthermore, a rank one perturbation of a matrix $A$ results in a rank two perturbation of $A^*A$. Hence, to study rank one perturbations of singular values one needs to consider rank two perturbations of Hermitian matrices. Rank one perturbations of singular values appear e.g. in the context of rank one updatability of the svd decomposition, see e.g.  \cite{Stange1,Stange2}.


\medskip

In noncommutative probability transformations of measures play an important role, e.g. they allow to define new convolutions \cite{bozejko_wysoczanski98, bozejko_wysoczanski01, wojakowski07}. 
Of particular interest are the transformations for which there exists an operator model, i.e.  there exists a deformation of the operator such that the corresponding deformation of the spectral measure (distribution w.r.t. a certain state) is the transformation in question.  
One of such measure transforms, called the $\mathbf{t}$-transform, was introduced by Bo\.zejko and Wysocza\'nski in \cite{bozejko_wysoczanski98, bozejko_wysoczanski01}. It was shown that the $\mathbf{t}$-transform $\mu_t$ of the Wigner law $d\mu_0(t)=\frac{1}{2\pi}\mathbf{1}_{[-2,2]}(t)\sqrt{4-t^2}$ can be obtained when the free creation and the free anihilation on the free Fock space become (rank one) deformed. Then $\mu_t$ is the distribution of the rank two deformation of the sum of creation and annihilation. The operator corresponding to $\mu_t$ has the form
$$
{J_t}=\begin{pmatrix} 0 & 1-t & & \\ 1- t & 0 & 1\\ & 1 & 0 & 1\\ &&1& \ddots &\ddots\\ && & \ddots\end{pmatrix} \quad  (t\in\Real),
$$
and can be considered as bounded operator in $\ell^2$. The plots of $\mu_t$ for different values of $t$, obtained numerically, can be seen in Figure \ref{densities}. 
\begin{figure}[htb]
\includegraphics[width=200pt]{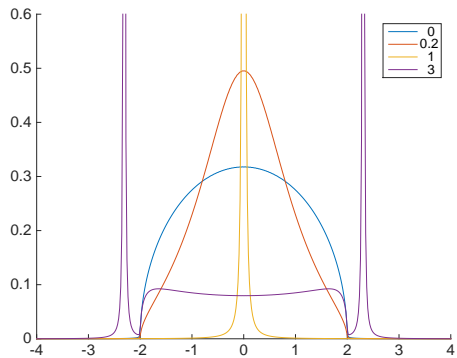}
\caption{Numerically obtained plots for densities $\mu_t$
}
\label{densities}\end{figure}
The result for the Wigner law suggests that the operator model for the $\mathbf{t}$-transform is related to the rank two deformation. 

\medskip
This article splits into three main parts. 
\begin{itemize}
\item 
Section 1 contains the main results on the Weyl function and spectrum of rank two perturbation of an operator. The results are first derived for an arbitrary perturbations of the form $A-s \D w u - t\D gh$, then the formula is analyzed in several instances. In particular we study the `diagonal' ($A-s \D u u - t\D {Au} {Au}$) and `anti-diagonal' ($A-s \D {Au} u - t\D u {Au}$) deformations.

\item In Section 2 we apply our general results in the linear algebra setting, i.e.\ to the finite dimensional operators. 
In particular, the characteristic polynomial of a rank two perturbation is computed. The results are then used to obtain a large parameter limits of the perturbed eigenvalues and large parameter limits of rank one perturbations of singular values.  

\item In Section 3 we apply the results to noncommutative probability framework, using the correspondence between self-adjoint operators and positive measures (their distributions). 
We show that the `anti-diagonal' rank two deformation provides an operator model of the measure transformation defined by Krystek and Yosida \cite{KrYos2004}, a generalization of the $\mathbf{t}$-transform, and we study its phase transition properties, by which we mean appearance or disappearance of discrete part of the measures. We also define and investigate the transformation that is related to the `diagonal' deformation. To the best of our knowledge this transform has not been known yet. 

\end{itemize}

\section{Weyl function for two-dimensional perturbations of operators}
\subsection{Preliminaries}

Through the whole paper $(\H,\seq{\cdot,\cdot})$ denotes a complex Hilbert space with a scalar product. The reader more interested in the finite dimensional part of this paper may consider $\H=\Comp^n$ with the euclidean inner product. A closed, densely defined operator is in such case a matrix, with  $\dom A=\Comp^n$ and the resolvent set is always nonempty. The results of the present section remain nontrivial in this case. By $B(\H)$ we denote the algebra of bounded operators on the Hilbert space $\H$, in particular we identify $B(\Comp^n)=\Comp^{n\times n}$.

\begin{defi}
Let $A$  be a closed, densely defined operator on $\H$ with nonempty resolvent set $\rho(A)$. We define the \textit{Weyl function of  $A$ with respect to  vectors $u, w \in \H$} by
\begin{equation}\label{Weyl function def}
 Q_{u,w} (z) := \<(z -A )^{-1}u,w\>, \quad z\in \rho(A).
\end{equation}
We abbreviate $Q_u(z):=Q_{u, u} (z)$ if $u=w$. 
\end{defi}

Clearly the function $Q_{u, w} (z)$ is holomorphic on $\rho(A)$. The behavior of  $Q_{u, w} (z)$ on the spectrum of $A$ might
be very complicated, especially in the operator case, see e.g. \cite{Greenstein,SWW2}, and is not the objective of the present paper. 


For $u, w, f\in\H$ we define the rank-one operator $\D wu$ by $({\D wu})f:=\seq{f,w}u$. 
The following Lemma exhibits crucial properties of a rank-one deformation of an invertible operator. In particular, it determines the values of $s$ for which the deformation $B+s\D{w}{u}$ is invertible.


\begin{lemma} \label{lem_inverse}
Let $B$ be a closed, densely defined, invertible operator, $u,w\in\H\setminus\set0$, $s\in\Comp$ and let $Q_{u,w}(z):=\<(z-B)^{-1}u,w\>$. Then
the operator
$$
B_s:=B+s\D{w}{u}
$$
is invertible if and only if $1+s \<B^{-1}u,w\>\neq 0$ and, in such case, the inverse is given by the formula
\begin{equation}\label{Btinv}
B_s^{-1}=B^{-1}-\frac{s}{1+s \<B^{-1}u,w\>}\D{B^{*-1}w}{B^{-1}u}.
\end{equation}
Moreover, if $z\in \rho(B)$ and $1+sQ_{u,w}(z)\neq 0$, then for every $\xi, \eta\in \H$ we have the formula
\begin{equation}\label{formula for Weyl general}
\<(z-B_s)^{-1}\xi, \eta\> = \frac{Q_{\xi\eta}(z)+s Q_{\xi\eta}(z)Q_{uw}(z)- sQ_{\xi w}(z)Q_{u\eta}(z)}{1+s Q_{uw}(z)},
\end{equation}
where we use the notation \eqref{Weyl function def} for the corresponding Weyl functions.  
\end{lemma}

\begin{proof}
Suppose $B_s=B+s\D{w}{u}$ is invertible and let $B_s^{-1}=A\in \B(\H)$ denote its inverse. Then,
$$
f=B_sAf=BAf+s\seq{Af,w}u,\quad f\in\H.
$$
In particular, for $f=A^{-1}B^{-1}u$ we have
$$
u+s\seq{B^{-1}u,w}u = A^{-1}B^{-1}u\neq 0,
$$
hence $1+s\seq{B^{-1}u,w}\neq0$. The other implication can be checked directly,  by showing that the operator on the right hand side of \eqref{Btinv} is the inverse of $B+s\D{w}{u}$.

Assuming $z\in \rho(B)$ and $1+sQ_{u,w}(z)=0$ and using (\ref{Btinv}) we get
\begin{eqnarray*}
&& \<(z-B_s)^{-1}\xi, \eta\> \\
&=& \left\<\left( (z-B)^{-1}-\frac{s}{1+s\<(z-B)^{-1}u,w\>}\D{(z-B)^{*-1}w}{(z-B)^{-1}u} \right)\xi,\eta \right\>\\
&=& \< (z-B)^{-1}\xi, \eta\> -\frac{s\<(z-B)^{-1}\xi,w\>\<(z-B)^{-1}u ,\eta\>}{1+s\<(z-B)^{-1}u,w\>} \\
 &=& Q_{\xi\eta} -\frac{sQ_{\xi w}Q_{u\eta}}{1+s Q_{uw}} = \frac{Q_{\xi\eta}+s Q_{\xi\eta}Q_{uw}- sQ_{\xi w}Q_{u\eta}}{1+s Q_{uw}}.
\end{eqnarray*}

\end{proof}

\subsection{The general perturbations of the form $A-s(\D{w}{u}) - t(\D hg)$}

The main object of our study is the sum of two rank-one deformations. 

\begin{thm}\label{Theorem 1} Let $A$ be a closed, densely defined operator on $\H$, $s, t\in \C$ and let
$$
A_{s,t} := A-s(\D{w}{u}) - t(\D hg), 
$$
for some nonzero vectors $u,w,g,h\in\H$, with $u\in\dom(A)$. 

\begin{enumerate}
\item[(i)] For $z\in\rho(A)\cap\rho(A_{s,0})$ we have that  $z\in\sigma(A_{s,t})$ if and only if
\begin{equation}\label{z in spectrum}
1+sQ_{uw}(z)+tQ_{gh}(z)+stQ_{gh}(z)Q_{uw}(z)-stQ_{gw}(z)Q_{uh}(z)=0.
\end{equation}

\item[(ii)] 
The Weyl function $Q^{s,t}_{u}(z):=\seq{(z-A_{s,t})^{-1} u,u}$ of the operator $A_{s,t}$

equals
\begin{equation}\label{Weyl function 1}
Q_{u}^{s,t}=\frac{Q_u+tQ_{u}Q_{gh}-tQ_{uh}Q_{gu}}{1+sQ_{uw}+tQ_{gh}+stQ_{gh}Q_{uw}-stQ_{gw}Q_{uh}}
\end{equation}
\end{enumerate}
\end{thm}
\begin{proof}

Fix $z\in\rho(A)\cap \rho(A_{s,0})$ and let $B=z-A$. Set $B_s = B+s\D{w}{u}$, 
so that
$$
z-A_{s,t} =  B+s\D{w}{u} + t\D hg= B_s+ t\D hg.
$$
Using the first part of Lemma \ref{lem_inverse} we have that $B_s+ t\D hg$ is invertible if and only if $1+t\<B_s^{-1}g,h\>\neq 0$. 
Applying (\ref{formula for Weyl general}) of Lemma \ref{lem_inverse}, we get 
\begin{equation}\label{douzupelnienia}
\<B_s^{-1}g, h\>
 = \frac{Q_{gh}+s Q_{gh}Q_{uw}- sQ_{g w}Q_{uh}}{1+s Q_{uw}}.
\end{equation}
Observe that, by Lemma \ref{lem_inverse}, the assumption $z\in \rho(A_{s,0})$ is equivalent to $1+s Q_{uw}\neq 0$.  It follows, that $1+t\<B_s^{-1}g,h\>\neq 0$ if and only if 
\[
1+\frac{tQ_{gh}+st Q_{gh}Q_{uw}- stQ_{g w}Q_{uh}}{1+s Q_{uw}}\neq 0,
\]
from which one gets (\ref{z in spectrum}). 

(ii) For $z\in \rho(A_{s,0})$ the operator $B_s$ is invertible, so applying Lemma \ref{lem_inverse} we obtain
\begin{eqnarray*}
Q^{s,t}_{u} (z)
&=& \< (z-A_{s,t})^{-1}u,u\> =\<(B_s + t\D{h}{g})^{-1}u,u\> \\
&=& \< B_s^{-1}u,u\> -\frac{t}{1+t\<B_s^{-1}g,h\>}\<(\D{B_s^{*-1}h}{B_s^{-1}g})u,u\>\\
&=&  \< B_s^{-1}u,u\>-\frac{t\< u,B_s^{*-1}h\>\<B_s^{-1}g,u\>}{1+t\<B_s^{-1}g,h\>}\\
&=& \< B_s^{-1}u,u\>-\frac{t\< B_s^{-1}u,h\>\<B_s^{-1}g,u\>}{1+t\<B_s^{-1}g,h\>}.\\
\end{eqnarray*}
Assuming $z\in \rho(A)$ we have that $B=z-A$ invertible, so by (\ref{formula for Weyl general}) we get formulas
\begin{eqnarray*}
\< B_s^{-1}u,u\> = \frac{Q_u}{1+sQ_{uw}}, \quad
\<B_s^{-1}g,u\> = \frac{Q_{gu}+sQ_{gu}Q_{uw}-sQ_{gw}Q_{u}}{1+sQ_{uw}}, \\
\< B_s^{-1}u,h\>  = \frac{Q_{uh}}{1+sQ_{uw}}, \quad \<B_s^{-1}g,h\> = \frac{Q_{gh}+sQ_{gh}Q_{uw}-sQ_{gw}Q_{uh}}{1+sQ_{uw}}. \\
\end{eqnarray*}
Using these we get
\begin{eqnarray*}
Q^{t,s}_{u} (z)
&=& \frac{Q_{u}}{1+s Q_{uw}}-\frac{t\< B_s^{-1}u,h\>\<B_s^{-1}g,u\>}{1+t\<B_s^{-1}g,h\>}\\
&=&  \frac{Q_{u}}{1+s Q_{uw}}-\frac{ t\frac{Q_{uh}}{1+ s Q_{uw}}\cdot  \frac{Q_{gu}+s Q_{gu}Q_{uw} -sQ_{gw}Q_{u}}{1+s Q_{uw}}}{1+t\frac{Q_{gh}+s Q_{gh}Q_{uw}
-sQ_{gw}Q_{uh}}{1+s Q_{uw}} }\\
&=& \frac{1}{1+s Q_{uw}}
\left(Q_u-\frac{ tQ_{uh}[Q_{gu}+s Q_{gu}Q_{uw} - sQ_{gw}Q_{u}]}{1+sQ_{uw}+tQ_{gh}+stQ_{gh}Q_{uw}-stQ_{gw}Q_{uh}}\right)\\
\end{eqnarray*}
The numerator in the brackets becomes
\begin{eqnarray*}
Q_u+sQ_uQ_{uw}+tQ_{u}Q_{gh}+stQ_{u}Q_{gh}Q_{uw}-stQ_{u}Q_{gw}Q_{uh}-tQ_{uh}Q_{gu}-stQ_{uh}Q_{gu}Q_{uw}\\+stQ_{uh}Q_{gw}Q_{u} = Q_u+sQ_uQ_{uw}+tQ_{u}Q_{gh}+stQ_{u}Q_{gh}Q_{uw}-tQ_{uh}Q_{gu}-stQ_{uh}Q_{gu}Q_{uw}
\end{eqnarray*}
hence we obtain the formula
\begin{equation}\label{Weyl function general}
Q^{s,t}_{u} (z)= \frac{Q_u+sQ_uQ_{uw}+tQ_{u}Q_{gh}+stQ_{u}Q_{gh}Q_{uw}-tQ_{uh}Q_{gu}- stQ_{uh} Q_{gu}Q_{uw}}{[1+s Q_{uw}][1+sQ_{uw}+tQ_{gh}+stQ_{gh}Q_{uw}-stQ_{gw}Q_{uh}]}
\end{equation}
After factoring out $1+s Q_{uw}$ in the numerator, we get (\ref{Weyl function 1}).
\end{proof}



\subsection{Perturbations for $\{u,w\}=\{g,h\}$}


We consider now several instances of Theorem \ref{Theorem 1}.  We will always study two kinds of rank-two perturbations: the `antidiagonal' and `diagonal', given respectively by the following formulas 
\begin{equation}\label{tildehat}
\widetilde{A_{s,t}}= A- s(\D{w}{u}) -t(\D{u}{w}),\quad  \widehat{A_{s,t}}= A- s(\D{u}{u}) -t(\D{w}{w}).
\end{equation}
By $\widetilde{Q_{u}^{s,t}}$ and  $\widehat{Q_{u}^{s,t}}$ we denote the  Weyl function of $\widetilde{A_{s,t}}$ and $\widehat{A_{s,t}}$, respectively. Both tilde and hat convention will be used later on in Theorem \ref{thm_gen} in the special case $w=Au$.

\begin{cor}\label{coruAAu} Let $A$ be a closed, densely defined operator and let $u,w\in\H$. 
\begin{enumerate}
\item[(i)]
For the operator 
$\displaystyle \widetilde{A_{s,t}}:= A- s(\D{w}{u}) -t(\D{u}{w})$ with $s,t\in\Comp$, 
the following hold:
\begin{enumerate}
\item[(i.1)] If $z\in \rho(A)\cap \rho(\widetilde{A_{s,0}})$ then $z\in \sigma(\widetilde{A_{s,t}})$ if and only if
\[
1+s Q_{u,w}(z)+tQ_{w,u}(z)+st Q_{u,w}(z)Q_{w,u}(z) -stQ_{u}(z)Q_{w}(z) =0.
\]

\item[(i.2)]
The Weyl function $\widetilde{Q_{u}^{s,t}}(z):=\<(z-\widetilde{A_{s,t}})^{-1} u,u \>$ for the operator $\widetilde{A_{s,t}}$ equals
$$
\widetilde{Q_{u}^{s,t}}=\frac{Q_{u}}{1+s Q_{u,w}+tQ_{w,u}+st Q_{u,w}Q_{w,u}
-stQ_{u}Q_{w}}.
$$
\end{enumerate}
\item[(ii)]
For the operator 
$\displaystyle \widehat{A_{s,t}} := A- s(\D{u}{u}) - t(\D{w}{w})$ with $s,t\in\Comp$, 
the following hold:
\begin{enumerate}
\item[(ii.1)] If $z\in \rho(A)\cap \rho(\widehat{ A_{s,0}})$ then $z\in \sigma(\widehat{A_{s,t}})$ if and only if 
\[
 \Leftrightarrow \quad (1+sQ_u(z))(1+tQ_w(z))=stQ_{u,w}(z)Q_{w,u}(z).
\]
\item[(ii.2)]
The Weyl function $\widehat{Q}_{u}^{s,t}:=\<(z-\hat{A}_{s,t})^{-1} u,u \>$ for this operator
equals
\[
\widehat{Q_{u}^{s,t}}=\frac{ Q_{u}(1+tQ_{w})- tQ_{u,w}Q_{w,u}}{(1+sQ_u)(1+tQ_w)- stQ_{u,w}Q_{w,u}}.
\]
\end{enumerate}
\end{enumerate}
\end{cor}

\begin{proof}
By substituting $g:=w$, $h:=u$ in Theorem \ref{Theorem 1} we get (i.1). For (i.2) observe, that, with this substitution, the numerator in (\ref{Weyl function general}) becomes
\begin{eqnarray*}
Q_u+sQ_uQ_{uw}+tQ_{u}Q_{wu}+stQ_{u}Q_{wu}Q_{uw}-tQ_{u}Q_{wu}- stQ_{u} Q_{wu}Q_{uw} &=& \\ 
Q_u(1+sQ_{uw})+tQ_{u}Q_{wu}(1++sQ_{uw})-tQ_{u}Q_{wu}(1+ sQ_{uw}), 
\end{eqnarray*}
from which (i.2) follows.

In a similar way, by putting $w=u$ in Theorem \ref{Theorem 1}, and then substituting $g=h=w$, we get (ii.1). For (ii.2) observe, that, with this substitution, the numerator in (\ref{Weyl function general}) becomes
\begin{eqnarray*}
Q_u+sQ_uQ_{u}+tQ_{u}Q_{w}+stQ_{u}Q_{w}Q_{u}-tQ_{uw}Q_{wu}- stQ_{uw} Q_{wu}Q_{u} &=& \\ 
(1+sQ_u)(Q_u+tQ_{u}Q_{w}-tQ_{uw}Q_{wu}), 
\end{eqnarray*}
while the denominator becomes 
\begin{eqnarray*}
(1+s Q_{u})(1+sQ_{u}+tQ_{w}+stQ_{w}Q_{u}-stQ_{wu}Q_{uw}) &=&\\ 
(1+s Q_{u})(1+tQ_w)-stQ_{wu}Q_{uw},
\end{eqnarray*}
from which (ii.2) follows.
\end{proof}


\subsection{Perturbations for $w=Au$}

In case  $w=Au$ the formulas can be simplified, according to the following lemma.

\begin{lemma}\label{QAuu}
Let $A$ be a closed and densely defined operator with nonempty resolvent set and let $u\in\dom(A)$, $\norm u=1$, $m=\<u,Au\>$. Then with $Q_{u,w}(z)=\seq{(z-A)^{-1}u,w}$ one has 
\begin{equation}\label{QuAu}
Q_{Au,u}(z)=zQ_{u}(z)-1,\quad z\in\rho(A),
\end{equation}\label{Qef}
\begin{equation}
Q_{u,Au}(z)=\frac{1}z(m+Q_{Au}(z)),\quad z\in\rho(A).
\end{equation}
\end{lemma}

\begin{proof}
To see the first equality observe that
\begin{eqnarray*}
Q_{Au,u}(z)&=& \<(z-A)^{-1}Au,u\>\\
&=& \<(z-A)^{-1}(A-z)e+z(z-A)^{-1}u,u\> \\
&=& -1+zQ_{u} (z).
\end{eqnarray*}
Transforming similarly $Q_{u,Au}(z)$ one obtains the second equality.
\end{proof}

Using this Lemma we can describe the rank-two deformations given in \eqref{tildehat} in the case $w=Au$. 
\begin{thm} \label{thm_gen}
Let $A$ be a closed, densely defined operator and let $u\in\dom(A)$, $\norm u=1$, $m=\<u,Au\>$.  
\begin{enumerate}
\item[(i)]
For the operator 
$\displaystyle \widetilde{A_{s,t}}:= A- s(\D{Au}{u}) -t(\D{u}{Au})$ with $s,t\in\Comp$, 
the following hold:
\begin{enumerate}
\item[(i.1)] If $z\in \rho(A)\cap \rho(\widetilde{A_{s,0}})$ then $z\in \sigma(\widetilde{A_{s,t}})$ if and only if
\[
(1-t)\left[1+\frac{s}{z}(m+Q_{Au})\right]+t(z+sm)Q_u(z)=0
\]
\item[(i.2)]
The Weyl function $\widetilde{Q_{u}^{s,t}}:=\<(z-\widetilde{A_{s,t}})^{-1} u,u \>$ 
is given by
\begin{equation} 
\frac{1}{\widetilde{Q_{u}^{s,t}}(z)}=\frac{1-t}{Q_{u}}\left(1+\frac{s}{z}(m+Q_{Au}(z)) \right) +t(z+sm)
\end{equation}

\end{enumerate}
\item[(ii)]
For the operator 
$\displaystyle \widehat{A_{s,t}} = A- s(\D{u}{u}) - t(\D{Au}{Au})$ with $s,t\in\Comp$, 
the following hold:
\begin{enumerate}
\item[(ii.1)] If $z\in \rho(A)\cap \rho(A_{s,0})$ then $z\in \sigma(\widehat{A_{s,t}})$ if and only if
\[
(z+stm) + sz(1-tm)Q_u + t(z+s)Q_{Au}=0.
\]
\item[(ii.2)]
The Weyl function $\widehat{Q_{u}^{s,t}}:=\<(z-\widehat{A_{s,t}})^{-1} u,u \>$ 
is given by
\begin{equation}
\frac{1}{\widehat{Q_{u}^{s,t}}}=s+\frac{1+tQ_{Au}}{(1-tm)Q_u(z)+ \frac{t}{z}(m+Q_{Au})}
\end{equation}
\end{enumerate}
\end{enumerate}
\end{thm}


\begin{proof}
(i) First observe that, by Lemma \ref{QAuu}, we have 
\[
1+s Q_{u,Au}+tQ_{Au,u}+st Q_{u,Au}Q_{Au,u}-stQ_{u}Q_{Au} =
\]
\[
= 1+\frac{s}{z}(m+Q_{Au})+t(zQ_{u}-1)+ \frac{st}{z}(zQ_u-1)(m+Q_{Au})-stQ_uQ_{Au} =
\]
\[
=(1-t)\left[1+\frac{s}{z}(m+Q_{Au})\right]+t(z+sm)Q_u. 
\]
Hence, by Corollary \ref{coruAAu}(i.1) applied to $w:=Au$ we get (i.1). Furthermore, 
as 
\[
\widetilde{Q_{u}^{s,t}}=\frac{Q_{u}}{1+s Q_{u,Au}+tQ_{Au,u}+st Q_{u,Au}Q_{Au,u}-stQ_{u}Q_{Au}}
\]
formula (i.2) follows. 

(ii) Similarly, observe that 
\[
Q_{u}(1+tQ_{Au})- tQ_{u,Au}Q_{Au,u} = (1-tm)Q_u+\frac{t}{z}(m+Q_{Au}),
\]
which together with Corollary \ref{coruAAu}(i.2) applied to $w:=Au$  shows (ii.1). Furthermore, as 
\[
\frac{1}{\widehat{Q_{u}^{s,t}}}=s+\frac{1+tQ_{Au}}{Q_{u}(1+tQ_{Au})- tQ_{u,Au}Q_{Au,u}}
\]
formula (ii.2) follows. 
\end{proof}


\subsection{Deformations of self-adjoint operators} In the self-adjoint case $A=A^*$ we can also simplify the formula for $Q_{Au}(z)$. 

\begin{lemma}\label{QAA} Let $A=A^*$ be a closed, densely defined operator, $u\in \dom(A)$, $\norm u=1$, $m=\<u,Au\>=\<Au,u\>$ and let
$Q_{u}(z)$ denote the Weyl function of $A$ with respect to $u$. Then
\begin{eqnarray*}
Q_{Au}(z) =  z^2Q_{u}(z) - z - m.
 \end{eqnarray*}
\end{lemma}

\begin{proof} By simple computations we get 
\begin{eqnarray*}
Q_{Au}(z) &=& \<(z-A)^{-1} Au,Au   \>\\
&=& \< A (z-A)^{-1} Au,u\>\\
&=& \< A(z-A)^{-1}(A-z)u,u\>+z\<A(z-A)^{-1}u,u\>\\
&=& -m+ z(zQ_{u}(z)-1).
\end{eqnarray*}
\end{proof}
Substituting  Lemma \ref{QAA} in  Theorem \ref{thm_gen} we get the following result. 
\begin{thm} \label{thm_sa}
Let $A=A^*$ be a closed, densely defined operator on a Hilbert space $\H$,  $u\in \dom(A)\subset \H$ with $\|u\|=1$, $m=\<u,Au\>$, and let $Q_{u}(z)$ denote the Weyl function of $A$ with respect to $u$. 
\begin{enumerate}
\item[(i)] 
For the operator 
$\displaystyle \widetilde{A_{s,t}}:= A- s(\D{Au}{u}) -t(\D{u}{Au})$ with $s,t\in\Comp$, 
the following hold:
\begin{enumerate}
\item[(i.1)] If $z\in \rho(A)\cap \rho(\widetilde{A_{s,0}})$ then $z\in \sigma(\widetilde{A_{s,t}})$ if and only if
\[
(1-s)(1-t) + [z(s-st+t)+stm]Q_u(z)=0.
\]
\item[(i.2)]
The Weyl function $\widetilde{Q_{u}^{s,t}}(z):=\<(z-\widetilde{A_{s,t}})^{-1} u,u \>$ 
is given by
\begin{equation} \label{eq_G_self_adjoint}
\frac{1}{\widetilde{Q^{s,t}_{u}} (z)} =  \frac{(1- s)(1-t)}{Q_{u}(z)} +
(s-st+t)z +stm;
\end{equation}
\end{enumerate}
\item[(ii)]
For the operator 
$\displaystyle \widehat{A_{s,t}} = A- s(\D{u}{u}) - t(\D{Au}{Au})$ with $s,t\in\Comp$, 
the following hold:
\begin{enumerate}
\item[(ii.1)] If $z\in \rho(A)\cap \rho(\widehat{A_{s,0}})$ then $z\in \sigma(\widehat{A_{s,t}})$ if and only if
\[
(1-st-tz-tm)+(s-stm+stz+tz^2)Q_u(z)=0.
\]
\item[(ii.2)]
The Weyl function $\widehat{Q_{u}^{s,t}}(z):=\<(z-\widehat{A_{s,t}})^{-1} u,u \>$ 
is given by
\begin{equation} \label{eq_G_self_adjointD}
\frac{1}{\widehat{Q_{u}^{s,t}}(z)} = s+\frac{1+t(z^2Q_{u}(z) -m-z)}{(1-tm+tz)Q_{u}(z)-t}.
\end{equation}
\end{enumerate}
\end{enumerate}
\end{thm}
\begin{proof}
Note that $m\in\Real$, due to $A=A^*$. Applying   Theorem \ref{thm_gen}(i) and Lemma \ref{QAA} we get the conclusion (i). Similarly, applying  Theorem \ref{thm_gen}(ii) and Lemma \ref{QAA} we get (ii).
\end{proof}


\begin{rem}
Note that the only  property  the inner product $\seq{\cdot,\cdot}$ that was used in  Lemma \ref{lem_inverse}, Theorem \ref{Theorem 1}, Corollary \ref{coruAAu}, Lemma \ref{QAuu}, Theorem \ref{thm_gen}   was sesquilinearity. Consequently, the definite inner product can be replaced everywhere by a form
$\seq{Hx,y}$, where $H\in B(\mathcal{H})$ and $H=H^*$  and the Weyl function can be replaced by the $H$-Weyl function
$$
Q^H_u(z) := \seq{H(z-A)^{-1}u,u}.
$$
Then,  Lemma \ref{QAA} and Theorem \ref{thm_sa} also follow under the assumption that $HA=A^*H$.
\end{rem}

\begin{rem}
The assumption that $z\in\rho(A_{s,0})$ in Theorem \ref{Theorem 1}, Corollary \ref{coruAAu}, Theorem \ref{thm_gen} and Theorem \ref{thm_sa} seems to be unavoidable in the operator case, though we are not able to give any example.
However, in the matrix case, thanks to the existence of the characteristic polynomial, this assumption may be simply removed, cf. Theorem  \ref{matrixcase}(a) below.

\end{rem}

\section{Application: perturbations  of spectra of matrices}
In this section we will deal with $\mathcal{H}=\Comp^n$ endowed with the standard inner product. Therefore, we will write $uw^*$
  instead of   $\D w u$, furthermore,
 $$
Q_{u,w}(z)=w^*(z-A)^{-1}u.
$$
By $\norm{u}$ and $\norm{A}$ we mean the euclidean norm of a vector $u$ and the induced matrix norm of the matrix $A$, respectively. 
By saying that a a statement holds for generic vectors $u_1,\dots,u_k,v_1^*,\dots,v_l^* \in\Comp^n$ we mean that there exists a nonzero polynomial of $(k+l)n$ variables such that the set of all (coordinates of) vectors $u_1,\dots,u_k,v_1^*,\dots,v_l^* $ for which the statement does not hold is a subset of a zero set of a nonzero polynomial in $(k+l)n$ variables.

\subsection{Parametric rank two perturbations and large scale limits of eigenvalues}
In this subsection  we show how the developed the present paper  techniques may be used to reveal the spectrum of rank two  perturbation of a matrix.
In particular, we will be interested in large parameter limits of the spectrum. We extend here some ideas from \cite{RW12} for finding the limits of rank one perturbations to the rank two case.
However, we will refrain from investigating the generic Jordan structure of rank two perturbations, as this problem was addressed in the paper \cite{Batzkeetal}. 

\begin{thm} \label{matrixcase} Let $A\in\Comp^{n\times n}$, $n\geq 2$, then the following statements hold.
\begin{itemize}
\item[(a)]
For all $u,w,g,h\in\Comp^n$  
the characteristic polynomial of $A-suw^* - t gh ^*$ $(s,t\in\Comp)$ equals
$$
\det(z- A) R_{s,t}(z),
$$
where
$$
R_{s,t}(z)=1+s Q_{u,w}(z)+tQ_{g,h}(z)+st Q_{u,w}(z)Q_{g,h}(z) -stQ_{g,w}(z)Q_{u,h}(z).
$$
\item[(b)] For generic $u,w^*,g,h^*\in\Comp^n$ the function
$$
q(z)=\det(z- A)(Q_{u,w}(z)Q_{g,h}(z) -Q_{g,w}(z)Q_{u,h}(z))
$$
is a polynomial of degree $n-2$ with simple roots only.
\item[(c)] For generic $u,w^*,g,h^*\in\Comp^n$ and all $\alpha,\beta\in\Comp\setminus\set0$
if $\Real\ni r\to+\infty$ then two eigenvalues of $A-(\alpha r) uw^* - (\beta r) gh ^*$ converge to infinity  as
$r\lambda_i+o(r)$, where $\lambda_i$ $(i=1,2)$ are eigenvalues of the matrix $\alpha uw^*+\beta gh^*$.
Furthermore, $n-2$ eigenvalues    converge to the zeros of the polynomial
$q(z)$. 
\end{itemize}
\end{thm}

\begin{proof}	
(a) By \cite{RW12} the characteristic polynomial of $A-suw^*$ equals 
$\det(z- A) (1+sQ_{u,w}(z))$. Repeating this once again for $tfg^*$ and  $A-suw^*$ substituted for  $suw^*$  and $A$ respectively, we obtain by formula (\ref{formula for Weyl general}) of Lemma \ref{lem_inverse} that the characteristic polynomial of $A-suw^* - t gh ^*$
equals
$$
\det(z- A)(1+sQ_{u,w}(z))\left[ 1+t \frac{Q_{gh}(z)+s Q_{gh}(z)Q_{uw}(z)- sQ_{g w}(z)Q_{uh}(z)}{1+s Q_{uw}(z)}\right]
$$ $$
=\det(z-A)R_{s,t}(z).
$$

(b) Without loss of generality we may assume that $A$ is in the Jordan canonical form
\begin{equation}\label{Ajordan}
A=\bigoplus_{j=1}^l J^{k_i}_{\lambda_i} ,\quad J_{\lambda_i}^{k_i}=\begin{pmatrix} \lambda_i & 1 &&\\ &\ddots &\ddots &\\ &&\ddots &1 \\ &&& \lambda_i \end{pmatrix}\in\Comp^{k_i\times k_i}.
\end{equation}
Note that  $\xi^* (z- J^{k_j}_{\lambda_j})^{-1} \eta$ is a polynomial expression in $(z-\lambda_j)^{-1}$ of degree less or equal to $k_i-1$ of the form
\begin{equation}\label{xe}
\xi^* (z- J^{k_j}_{\lambda_j})^{-1} \eta = \sum_{i=1}^{k_j} \bar \xi_i\eta_i (z-\lambda_j)^{-1}+ \text{h.o.t.}  
\end{equation}
Consequently,
$$
R_\infty(z)=Q_{g,h}(z)Q_{u,w}(z) - Q_{g,w}Q_{u,h}(z)
$$
 is a polynomial expression in  $(z-\lambda_j)^{-1}$ ($j=1,\dots ,k$), without the constant term and without terms of degree one. 
 The statement is equivalent to saying that the above expression has nonzero terms of order two for generic $u,w^*,g,h^*$.
 Note that the coefficient at $(z-\lambda_{j'})^{-1}(z-\lambda_{j''})^{-1}$ is a polynomial expression in the entries of  
$u,w^*,g,h^*$, cf. \eqref{xe}. Hence, to finish the proof of (b) it is enough to show that for some $j',j''$ and for some 
  $u,w^*,g,h^*$ the coefficient at $(z-\lambda_{j'})^{-1}(z-\lambda_{j''})^{-1}$ is nonzero. 
  And so let 
$$
u=w=(1,0,0_{1,n-2})^\top,\quad g=h=(0,1,0_{1,n-2})^\top,\quad j'=1,\quad j''=\begin{cases} 1 & : k_1\geq1\\ 2 &: k_2 =1\end{cases}. 
$$
Then, according to \eqref{xe}, 
$$
R_\infty (z)=(z-\lambda_{j'})^{-1} (z-\lambda_{j''})^{-1}   +\text{ h.o.t. }
$$

%
(c) First observe that the matrix $r^{-1}(A-(\alpha r) uw^* - (\beta r) gh ^*)$ converges with $r\to\infty$ to $\alpha uw^*- \beta gh^*$. For generic $u,w,g,h$ the matrix $\alpha uw^*- \beta gh^*$ is of rank two for all $\alpha,\beta\in\Comp\setminus\set0$. Hence, there are two eigenvalues $\lambda_j(r)$ of $r^{-1} (A-(\alpha r) uw^* - (\beta r) gh ^*)$ that converge to the eigenvalues  $\lambda_j$ of $\alpha uw^*- \beta gh^*$. Note that $r\lambda_j(r)$ ($j=1,2$) is an eigenvalue of   $A-(\alpha r) uw^* - (\beta r) gh ^*$, which finishes the proof of the part of the statement concerning the eigenvalues converging to infinity.

To prove the second statement note that the eigenvalues of the perturbed matrix are by (b) zeros of the function $r^{-2}R_{\alpha r,\beta r}(z)$. The function $r^{-2}R_{\alpha r,\beta r}(z)$ converges locally uniformly to $R_\infty(z)$ with $r\to+\infty$. The claim follows now directly from the Rouche theorem.

\end{proof}

%

\begin{rem}[{Phase transition under small parameter perturbations}] \label{Phasetr}

 Let
$$
A_s=A-suw^* - s gh ^*, \quad s>0,
$$
be such that the spectrum of $A_s$ is symmetric with respect to the imaginary axis. Such situation happens if, for example, $A_s$ are Hamiltonian or real skew-symmetric $A_s=-A_s^\top$.  
Note that in such case the ODE $x'=A_sx$ is stable if and only if the spectrum of $A_s$ is purely imaginary.  Assume also that the spectrum of $A$ is purely imaginary.
We address the following question: what are then the necessary and sufficient conditions for the spectrum of $A_s$ to be contained in the imaginary axis for small values of the parameter $s$? We restrict ourselves to local investigations, observing the behavior of one particular eigenvalue $\lambda_0$ of $A$ under rank two peturbation. Assume that  $A$ has a Jordan chain of length two at $\lambda_0$, which is a typical situation for a phase transition, see e.g. \cite{SWW1}.
Consider the following Laurent expansions
$$
Q_{uw}(z)+Q_{gh}(z)= \sum_{j=-2}^\infty a_j (\lambda_0-z)^j,
$$ 
$$
Q_{u,w}(z)Q_{g,h}(z) -Q_{g,w}(z)Q_{u,h}(z)= \sum_{j=-2}^\infty b_j (\lambda_0-z)^j.
$$
We assume  that $a_{-2}\neq0$, which, in case of $A$ being Hamiltonian or skew symmetric and having a Jordan chain of length two,  is a generic assumption on $u,w^*,g,h^*$. The equation 
$R_{s,s}(z)=0$ from Theorem \ref{matrixcase}(a)
 takes the form
$$
1+s \sum_{j=-2}^\infty a_j (\lambda_0-z)^j + s^2  \sum_{j=-2}^\infty b_j (\lambda_0-z)^j =0.
$$
Observe that the solutions in $z$ of the above equation can be expanded for $s>0$ in the Puiseux series
$$
z_\pm(s)=\lambda_0\pm\sum_{j=1}^\infty c_j ( s^{1/2} )^j, 
$$ 
with
$$
c_1=\frac{-1}{a_{-2}},
$$
see also e.g. \cite{kato}.
Therefore, the necessary and sufficient conditions for $z(s)\in\ii\!\Real$ for small $s$  is that $a_{-2}\in\ii\!\Real$. 
Observe that this condition can be checked directly by numerical methods. Namely, $a_{-2}\in\ii\!\Real$ if and only if 
$$
Q_{uw}(\lambda_0\pm \ii\varepsilon)+Q_{gh}(\lambda_0+\pm \varepsilon)\in\Real
$$
for $\varepsilon$ sufficiently small. 

\end{rem}

\subsection{Parametric rank one perturbations and large scale limits of singular values}

Now we move to study of singular values of rank {\em one} perturbations of matrices. We refer the reader to \cite{Thompson76} for a non-parametric approach to the problem and to \cite{	Stange1,Stange2} for applications in numerical methods. As in the previous subsection we will be interested in the limits of the singular values for large values of the perturbation parameter. 
First observe the following fact. Let $u,v\in\Comp^n$, $\norm v=\norm u =1$, without loss of generality we may assume that $\tau>0$. Then
with $w=B^*v$ and $A=B^*B$  we have 
\begin{equation}\begin{gathered}
(B-\tau vu^*)^*(B-\tau vu^*) = B^*B - \tau B^*vu^*-   \tau uv ^*B + \tau^2 uv^*vu^*\\
 =  B^*B  -  \tau uw^*- \tau( w- \tau u)u^*,
\end{gathered}\label{BBB}
\end{equation}
which is a rank two perturbation of $B^*B$ unless $B^*v$ is a scalar multiple of $u$. Since in that case the usual perturbation theory for Hermitian matrices may be applied we will always assume that  $B^*v$ is not a scalar multiple of $u$. 
Such setting allows us to apply Theorem \ref{matrixcase}(a) with $A=B^*B$, $s= t=\tau$, $g=w-\tau u$, $h=u$. Here, this only one time in the paper, we violate the convention that the vectors $u,v,g,h$ do not depend on the parameters $s,t$. Therefore, we cannot use Theorem \ref{matrixcase}(c) directly but we need a  separate calculation. Furthermore, note that the above rank two perturbation of $B^*B$ is of special type. Namely, assume that $\rank B=k\leq n-2$. Then a generic rank two perturbation of $B^*B$ will be of rank
$k+2$. However, $\rank(B^*B  -  \tau uw^*- \tau( w- \tau u)u^*)=\rank(B-\tau vu^*)=k+1$ for generic $u,v$. 

\begin{thm}\label{svd}
Let $B\in\Comp^{m\times n}$  let $u,v\in\Comp^n$ be of norm one and such that $B^*v$ is not a scalar multiple of $u$. Let also 
 $\sigma_1(\tau)\geq \dots \geq \sigma_n(\tau)$ denote the singular values of $B-\tau vu^*$.
Then the following statements hold. 
\begin{itemize}
\item[(a)]
The characteristic polynomial of $(B-\tau vu^*)^*(B-\tau vu^*)$ $(\tau>0)$ equals  
$$
R_\tau(x)\det(x-B^*B),\quad \text{for}\quad x\in\Real,
$$
where
$$
R_\tau(x)= 1+ 2\tau \RE (  Q_{u,B^*v}(x) ) +\tau^2( -Q_{u,u}(x) + |Q_{u,B^*v}(x)|^2 -Q_{B^*v,B^*v}(x)Q_{u,u}(x)  ) .
$$
\item[(b)] The polynomial
$$
q(x)=\left( -Q_{u,u}(x) + |Q_{u,B^*v}(x)|^2 -Q_{B^*v,B^*v}(x)Q_{u,u}(x)  \right)  \det(x-B^*B)
$$
is of degree $n-1$ and has positive zeros $z_1,\dots,z_{n-1}$ satisfying
    $$
 \sigma_1(0)> \sqrt{z_1}\geq  \sqrt{z_2}\cdots \geq \sqrt{z_{n-1}}\geq 0.
 $$
\item[(c)] As $\tau\to+\infty$ one has
\begin{equation}\label{sigma1}
\sigma_1(\tau)=\tau+o(\tau),
\end{equation}
and
\begin{equation}\label{sigman}
\sigma_j(\tau) \to  \sqrt{z_{j-1}} ,\quad j=2,\dots,n.
\end{equation}
\end{itemize}
\end{thm}


\begin{proof}
 For the whole proof we fix arbitrary $u,v$ of norm one. 
 
 (a) Observe that
 \begin{eqnarray*}
R_\tau(x)&=&1+s Q_{u,w}(x)+tQ_{g,h}(x)+st Q_{u,w}(x)Q_{g,h}(x) -stQ_{g,w}(x)Q_{u,h}(x)\\
&=&   1+ \tau  Q_{u,w}(x)+\tau Q_{w-\tau u,u}(x)+\tau^2 Q_{u,w}(x)Q_{w-\tau u,u}(x)\\
& -& \tau^2Q_{w-\tau u,w}(x)Q_{u,u}(x)\\
&=& 1+\tau  Q_{u,w}(x)+\tau Q_{w,u}(x) -\tau^2 Q_{u,u}(x)+\tau^2 Q_{u,w}(x)Q_{w,u}(x) \\
&-& \tau^3 Q_{u,w}(x)Q_{u,u}(x) - \tau^2 Q_{w,w}(x)Q_{u,u}(x) +\tau^3 Q_{u,w}(x)Q_{u,u}(x).
\end{eqnarray*}
As we know that $R_\tau(x)$ has real roots only, we may consider $x\in\Real$, which simplifies the above to
\begin{equation}\label{Rtau}
R_\tau(x)= 1+2\tau \RE (  Q_{u,w}(x) ) +\tau^2( -Q_{u,u}(x) + |Q_{u,w}(x)|^2 -Q_{w,w}(x)Q_{u,u}  ) .
\end{equation}

(b) Like in the proof of Theorem  \ref{matrixcase} (b) (see therein for  details) we see that
$$
R_\infty(x):=( -Q_{u,u}(x) + |Q_{u,B^*v}(x)|^2 -Q_{B^*v,B^*v}(x)Q_{u,u}(x)  )
$$
is a polynomial expression in $(x-\lambda_j)^{-1}$ $(j=1,\dots, n)$ with the coefficients depending polynomially on the entries of $u$ and $v^*$.
 The statement is equivalent to saying that the above expression has nonzero terms of order one.
 This results from the fact that the matrix $A=B^*B$ is Hermitian and hence  $-Q_{u,u}(x)$ is a polynomial in $(x-\lambda_j)^{-1}$ $(j=1,\dots, n)$ of degree one and $ |Q_{u,B^*v}(x)|^2 -Q_{B^*v,B^*v}(x)Q_{u,u}(x)$ is a polynomial  in $(x-\lambda_j)^{-1}$ $(j=1,\dots, n)$ that has only terms of degree two. 

By a continuity argument all zeros of  $q(x)$ are nonnegative. Now assume that $x>\sigma_1^2(B)$, which implies that the matrix $(x-B^*B)^{-1}$ is positive definite and the quadratic form 
$[u,w]=Q_{u,w}(x)$ is an inner product. By the Cauchy-Schwartz inequality 
 we have
$$
R_\infty(x)< |Q_{u,B^*v}(x)|^2 -Q_{B^*v,B^*v}(x)Q_{u,u}(x) \leq 0.
$$
Hence, $q(x)\neq 0$.

(c) Observe that according to \eqref{BBB} there is one eigenvalue of $\tau^{-2}(B+\tau vu^*)^*(B+\tau vu^*)$ converging with $\tau\to\infty$ to $u^*u=1$.  Hence,  there is one singular value of $B+\tau vu^*$ converging to infinity as $\tau+{o}(\tau)$. 

Analogously as in Theorem  \ref{matrixcase} (b) we show that the remaining $n-1$ eigenvalues of 
$(B+\tau vu^*)^*(B+\tau vu^*)$ converge to the generically positive zeros $z_1,\dots,z_{n-1}$ of the polynomial $q$.
\end{proof}

\begin{rem}\label{Bvu}
Note that for proving statement \eqref{sigma1} we did not need the assumption that $B^*v$ is not a multiple of $u$. 
\end{rem}

\begin{rem}  Our claim is that the convergence  in \eqref{sigman} is generically linear, numerical simulation confirming this claim may be seen in Figure \eqref{svfig}. Nonetheless, 
in the case $u^*B^{-1}v=0$, which may be seen as a structure assumption, several sub-cases occur obscuring the general picture.
\end{rem}

\begin{thm} Let $B\in\Comp^{n\times n}$ $(n\geq2)$ be invertible,  let $u,v\in\Comp^n$ be of norm one and let  $\sigma_n(\tau)$ denote the smallest singular value  of $B-\tau vu^*$. Then 
 $\sigma_n(\tau)$ converges to zero  with $\tau\to\infty$ if and only if $u^*B^{-1}v=0$. In such case the convergence is linear, i.e. 
$\sigma(\tau)=\norm{{B^{-1}u}}^{-2}\tau^{-1}+o(\tau^{-1})$. 
In the opposite case, if  $\sigma_n(\tau)$ does not converge to $0$, then it converges to the  reciprocal of the largest singular value of
$$
B_\infty:=B^{-1} +\frac{1}{ u^*B^{-1} v} B^{-1}vu^*B^{-1}.
$$
\end{thm}

\begin{proof}
Observe first that since $B$ is invertible we have by Lemma \ref{lem_inverse} that 
\begin{equation}\label{Binf}
(B+\tau uv^*)^{-1}=B^{-1} -\frac{\tau}{1-\tau u^*B^{-1} v} B^{-1}vu^*B^{-1},\quad \tau\neq (u^*B^{-1} v)^{-1}.
\end{equation}

Assume that $u^*B^{-1} v\neq 0$.  Then the right hand side of \eqref{Binf} clearly converges with $\tau\to\infty$ to 
$B$.  In particular, the largest singular value of $(B+\tau uv^*)^{-1}$ converges to the largest singular value of $B_\infty^{-1}$. Note that as $\rank B_\infty\geq n-1$ the largest singular value of $(B+\tau uv^*)^{-1}$ is nonzero.  In consequence, the claim is proved.  

If $u^*B^{-1} v= 0$
then we apply Theorem \ref{svd} formula \eqref{sigma1} (cf. Remark \ref{Bvu}) to
$$
\tilde B=B^{-1},\quad \tilde v=\frac{B^{-1}v}{\norm{B^{-1}v}},\quad \tilde u=\frac{B^{-1}u}{\norm{B^{-1}u}}.
$$
Then the largest singular value of $(B+\tau uv^*)^{-1}$  behaves as $\norm{{B^{-1}u}}^2\tau+o(\tau)$   and consequently the smallest singular value of $B-\tau uv^*$ converges to zero as $\norm{{B^{-1}u}}^{-2}\tau^{-1}+o(\tau^{-1})$ with $\tau\to\infty$.

\end{proof}

\begin{cor}
If $B\in\Comp^{n\times n}$ $(n\geq2)$ is invertible then the $\norm\cdot_2$-condition number of rank one perturbation $B-\tau vu^*$ with $\norm u=\norm v=1$ with $\tau\to\infty$ equals
$\norm{B_\infty}^{-1}\tau+o(\tau)$ if $u^*B^{-1}v\neq 0$ and $\norm{{B^{-1}u}}^{2}\tau^2+o(\tau^2)$ in the opposite case. 
\end{cor}

\begin{figure}[ht]
\caption{The singular values of a $A+\tau uv^*$ (left figure) and the log-log plot of their linear convergence to the final location (right figure, the line $y=-x$ plotted in black for reference). The matrix $A\in\Comp^{10\times 10}$ and the vectors $u,v\in\Comp^{10}$ are random.}\label{svfig}
\includegraphics[width=200pt]{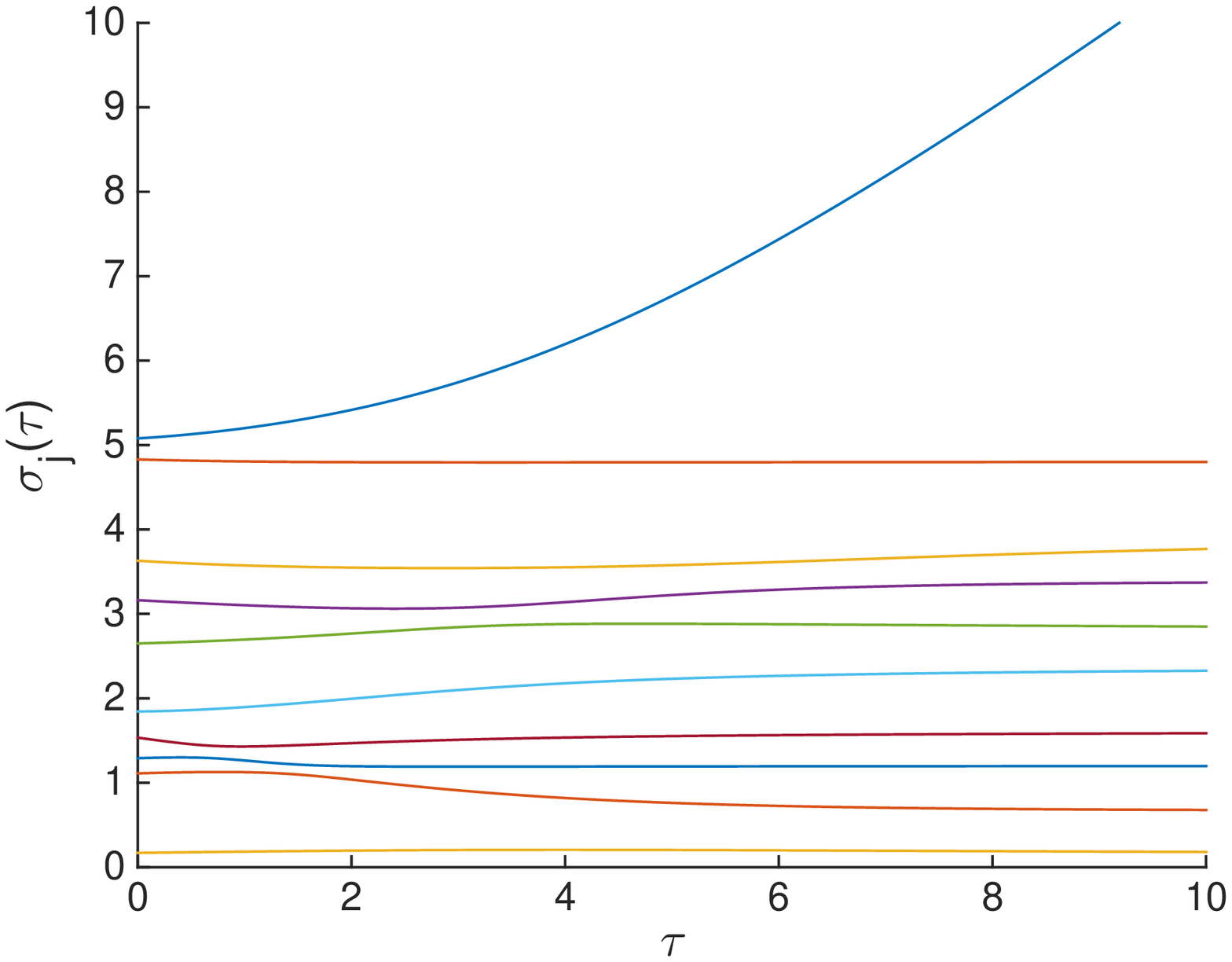}
\includegraphics[width=200pt]{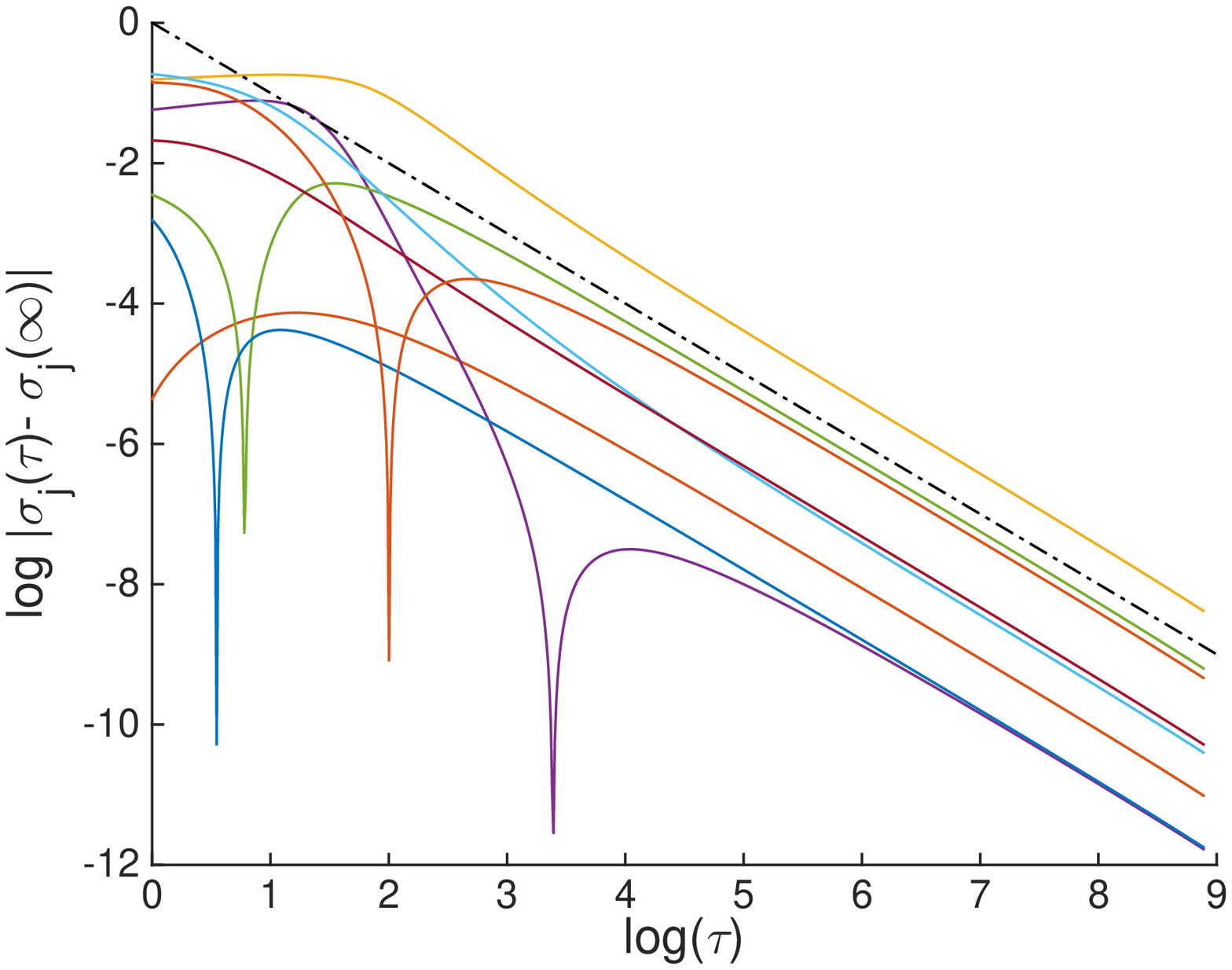}
\end{figure}

\subsection{Interlacing property and phase transition}

In this subsection we will be interested in the following property: the spectra of $A$ and $A_{s,t}=A-s uw^* - tgh^*$ are real and simple and between any two consecutive eigenvalues of one of the matrices there exists exactly one eigenvalue of another one. In such case we say that the spectra of $A$ and $A_{s,t}$ {\em interlace}. The general case was solved by Donat and Mulet in \cite{DM}. Below we show a simplified sufficient condition for $A=A^*$ and $\displaystyle \widetilde{\A_{s,t}}:= A- s(\D{Au}{u}) -t(\D{u}{Au})$ with real parameters $s,t$.
And so let
$A=A^*$ and let $\lambda_1, \dots , \lambda_n$ be the eigenvalues of $A$.
The Weyl function $Q_u(z):=\<(z-A)^{-1}u,u\>$ ($\norm u=1$)
has the form
\begin{equation}\label{Weylc}
Q_u(z)=\sum_{j=1}^{n}\frac{c_j}{z-\lambda_j},\quad \sum_{j=1}^nc_j=1.
\end{equation}
where the coefficients $c_j$ are nonnegative.
Furthermore, for generic $u\in\Comp^n$ they are positive.
 First note that  if  $(1-s)(1-t)>0$, then by the Remark \ref{remark on s,t real}, the spectrum  $\sigma(\widetilde{\A_{s,t}})$ is purely real but does not necessarily interlace the spectrum of $A$. We will present now a different sufficient condition for realness of spectrum and the interlacing property.

\begin{thm}\label{tinterlace}
Let $A=A^*\in\Comp^n$ has only simple eigenvalues and let $u\in\Comp^n$, $\norm u=1$ be such that  all coefficients $c_j$ $(j=1,\dots n)$ in \eqref{Weylc} are nonzero, i.e. $u$ is a generic vector. Then for $s,t\in\Real$ such that $st\neq s+t$ and
\begin{equation}\label{xolambda}
\frac{-stm}{st-s-t} < \min_{j=1,\dots,n}\lambda_j \quad \text{or}\quad \max_{j=1,\dots,n}\lambda_j <\frac{-stm}{st-s-t}
\end{equation}
the eigenvalues of $\widetilde{\A_{s,t}}:= A- s(\D{Au}{u}) -t(\D{u}{Au})$ are necessarily real
and interlace the eigenvalues of $A$.
\end{thm}

\begin{proof}
By Theorem \ref{thm_sa} the characteristic polynomial of $\widetilde{\A_{s,t}}$ equals
\[
\det(zI-A)[(1-s)(1-t)+[z(s+t-st)+stm]Q_u(z)]=0, 
\]
which has only real solutions $z=x\in \R$. Equivalently, for $x\notin\sigma (A)$ one has that   $x\in\sigma(\widetilde{\A_{s,t}})$ if and only if
\begin{equation}\label{Example diagonal real}
\sum_{j=1}^{n}\frac{c_j}{x-\lambda_j}= \frac{-(1-s)(1-t)}{(s+t-st)x+stm}.
\end{equation}
The left hand side is a decreasing rational function with real and simple poles. The right hand side is a hyperbola with the singularity in $\displaystyle x_0:=\frac{-stm}{st-s-t}$. Hence, between two consecutive eigenvalues of $A$ there is necessarily precisely one eigenvalue of $\widetilde{\A_{s,t}}$, see Figure \ref{finterlace}(left) for an illustration. The remaining eigenvalue of $\widetilde{\A_{s,t}}$ is necessarily real, otherwise its complex conjugate would be an eigenvalue as well, violating condition on the number of eigenvalues. It is either  greater than $\max\sigma(A)$ or smaller than $\min\sigma(A)$ and the interlacing property holds.

\end{proof}

\begin{rem}  Note that by elementary calculation $st-s-t\neq 0$ if $(1-s)(1-t)>0$. Though the latter condition implies realness of spectrum of  $\widetilde{\A_{s,t}}$ as well, it does not imply the intertwining of eignenvalues, as will be shown in Example \ref{einterlace}. 
\end{rem}

\begin{rem}
Note that in contrary to \cite{DM} the assumptions do not involve the coefficients $c_j$ but only the location of eigenvalues $\lambda_j$ and the coefficient $m$. 
\end{rem}

\begin{ex}\label{einterlace}

Let $A=\text{diag}(1,2,3,4)$, $u=0.5\cdot[1,1,1,1]^\top$.
For $s=1.1$, $t=1.2$ we have $x_0=-3.37$ and the assumptions of Theorem \ref{tinterlace} are satisfied. The spectrum of the perturbed matrix equals
$$
\sigma(\widetilde{\A_{s,t}})=\set{-3.25, 1.38, 2.50, 3.61}
$$
and it clearly interlaces with the eigenvalues of $A$. The functions  from equation \eqref{Example diagonal real} are plotted in Figure \ref{finterlace} (left).

Let now $s=-2$, $t=-3$.  Then $x_0=1.36$ so \eqref{xolambda} is not satisfied. 
The spectrum  of the perturbed matrix equals
$$
\sigma(\widetilde{\A_{s,t}})=\set{0.86,   2.16, 3.38, 16.10}
$$
 and clearly between $1$ and $2$ there is no eigenvalue of $\widetilde{\A_{s,t}}$. 
The functions in \eqref{Example diagonal real} are plotted in Figure \ref{finterlace} (right).
However, note that the spectrum of $\widetilde{\A_{s,t}}$ is real. Hence, \eqref{xolambda} is not a necessary condition for realness of the spectrum.

\begin{figure}\caption{
The hyperbola (blue) and rational function (red) from equation \eqref{Example diagonal real} for two different choices of the parameters $(s,t)$. The eigenvalues of $\widetilde{\A_{s,t}}$ marked with black crosses. See example \ref{einterlace} for details. }\label{finterlace}

\includegraphics[width=200pt]{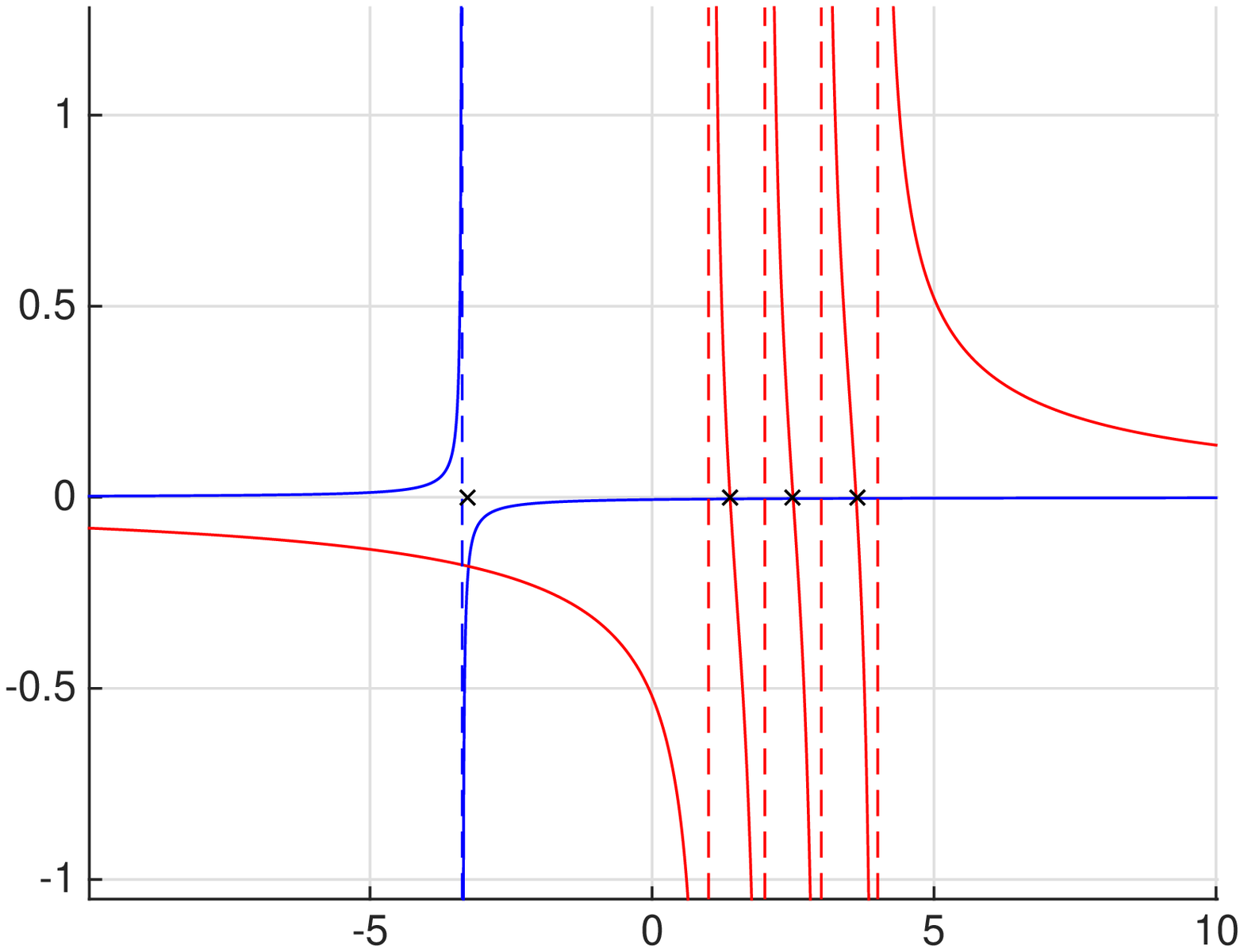}
\includegraphics[width=200pt]{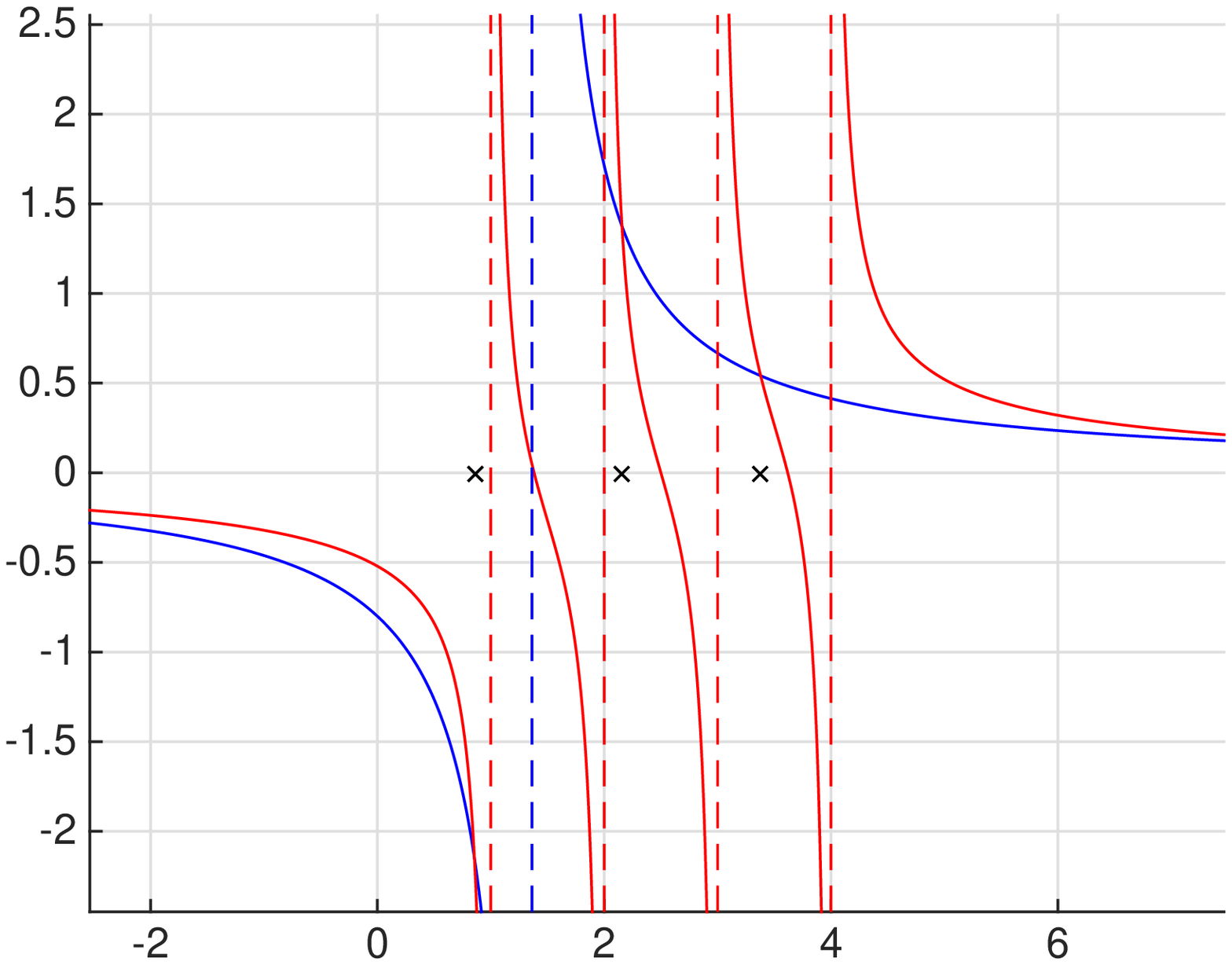}
\end{figure}

Furthermore, for  $s=1.1$, $s=0.9$ we get $x_0=-2.45$, hence the assumptions of Theorem \ref{tinterlace} are satisfied, while $(1-t)(1-s)<0$.
Hence, $(1-t)(1-s)>0$ is also not a necessary condition for realness of the spectrum.  
\end{ex}


\section{Applications in classical and free probability}

\subsection{Transformations of probability measures}
One of the goals of the paper was to provide an operator model of the $\mathbf{t}$-transform of probability measures, defined by Bo\.zejko and Wysocza\'nski \cite{bozejko_wysoczanski98} and extended to  \U-transform  by Krystek and Yoshida in \cite{KrYos2004}. For this aim we recall the definition of the Cauchy transform of a probability mesure and of the distribution of an operator with respect to a vector state.

\begin{defi}[Cauchy transform]\label{def_cauchy}
For a given probability measure $\mu $ on the real line $\R$,
its \textit{{Cauchy transform}}  is a holomorphic function $G_\mu(z)$ defined in the upper half-plane  $ \C^+=\{z\in \C: {\rm Im}\ z >0\}$,  as 
\[
G_{\mu}(z) = \int_{-\infty }^{+\infty } {\frac{d\mu (x)}{z-x}}.
\]
\end{defi}
\begin{defi}[Distribution of operator]\label{def_distribution_of_A}
Let $A$ be a self-adjoint operator (possibly unbounded, but for simplicity of the definition we shall assume $A\in \B(\H)$) and let $u\in \H$ be a normalized vector $\|u\|=1$ (in the domain of $A$ for unbounded $A$). Consider the vector state $\vp_u$ (i.e. positive, normalized functional) on $\B(\H)$, defined by  $\vp_u(B):=\<Bu,u\>$ {for} $B\in \B(\H)$. The spectral theorem implies, that there exists the unique probability measure $\mu_A$ on $\R$, such that 
\begin{equation}\label{phiz}
\vp_u((z-A)^{-1})=\int_{-\infty}^{\infty} \frac{ d\mu_A(x)}{z-x}, \quad  z\in\Comp^+.
\end{equation}
This measure called the \textit{distribution of $A$ with respect to $\vp_u$}.
\end{defi}

Note that for $z\in \C^{+}$ we have 
\begin{equation}\label{GQ}
G_{\mu_A}(z)=\int_{-\infty}^{\infty} \frac{\mu_A(dt)}{z-t}=\vp_u((z-A)^{-1})= \<(z-A)^{-1}u,u\> =Q_u(z),
\end{equation}
so the Weyl function of $A$ (w.r.t.\ $u$) equals the Cauchy transform of $\mu_A$ (i.e. of the distribution of $A$ w.r.t.\ $\vp_u$). 

\medskip
Let us recall the constructions of the {\bf t} and {\U}-transforms of probability measures. 
The main tool for this is the Nevanlinna theorem (\cite{Aaronson} Theorem 6.2.1.), 
which asserts that a function $F:\C^+\mapsto \C^+$ is the reciprocal of the Cauchy transform of a (uniquely determined) probability measure $\mu$ on the real line, i.e.\ $F(z)=\frac1{G_{\mu}(z)}$, if and only if there exist a positive measure $\rho $ and a constant $\alpha \in \R$ such that for
$z\in\Comp^+$
\begin{equation}\label{NevC}
F(z) = \alpha + z + \int_{-\infty }^{+\infty } {1+xz\over x-z} d\rho(x).
\end{equation}
Let $p, q\in \R$ with $q\geq 0$, then, given a probability measure $\mu$ on $\R$ with finite first moment $\displaystyle m:=\int x d\mu(x)$ and with $F(z)=\frac1{G_{\mu}(z)}$ as in \eqref{NevC}, one can show directly that the function 
\begin{equation*}
F_{p,q}(z):=\frac{q}{G_{\mu}(z)} + (1-q)z +(q-p)m =  \frac{q}{G_{\mu}(z)} + \frac{1-q}{G_{\delta_{cm}}(z)}
\end{equation*}
satisfies the condition \eqref{NevC} with $\alpha_{p,q}=q\alpha-(q-p)m$ and $\rho_{p,q}=\alpha\rho$. Consequently, 
 there exists the unique probability measure $U^{p,q}(\mu)$,  satisfying
\[
 F(z)=\frac{1}{G_{U^{p,q}(\mu)} (z)}=\alpha_{p,q} + z + \int_{-\infty }^{+\infty } {1+xz\over x-z} d\rho_{p,q}(x).
\]

\begin{defi}\label{K-Y-transform}
Let $p, q\in \R$ with $q\geq 0$, then for a probability measure $\mu$ on $\R$ with the finite first moment $m:=\int xd\mu(x)$, the \textit{\U-transform} of $\mu$ with the parameters $(p,q)$ is the unique probability measure $U^{p,q}(\mu)$ satisfying 
\begin{equation} \label{u-transform} 
\frac{1}{G_{U^{p,q}(\mu)} (z)}= \frac{q}{G_{\mu}(z)} + (1-q)z +(q-p)m.
\end{equation} 
In the case $p=q=\tau$ the \U-transform coincides with  the $\mathbf{t}$-transform, defined by 
\begin{equation} \label{t-transform}
\frac{1}{G_{\mu_\tau} (z)}= \frac{\tau}{G_{\mu}(z)} + (1-\tau)z,
\end{equation} 
see \cite{bozejko_wysoczanski98,bozejko_wysoczanski01}.\end{defi}

Note the following simple examples: $U^{0,1}(\mu)=\mu$, $U^{p,0}(\mu)=\delta_{pm}$.
Furthermore, both  \U- and $\bf t$-transform coincide on the set of all probability measures with vanishing first moment (i.e. with $m=0$). 


\begin{rem}\label{contfr} It is instructive to observe that the $\mathbf{U}$-transform modifies only the first two parameters in the continuous fraction expansion of $G_\mu(z)$. Namely, if 
\begin{equation}\label{Cauchy transform continued fraction}
G_{\mu}(z) = {1\over\displaystyle z-a_0- {\strut b_0\over\displaystyle
z-a_1- {\strut b_1\over\displaystyle z-a_2-{\strut b_2\over\ddots}}}}
\end{equation}
then 
\begin{equation}\label{U-transform of continued fraction}
G_{U^{p,q}(\mu)}(z) = {1\over\displaystyle z-pa_0- {\strut qb_0\over\displaystyle
z-a_1- {\strut b_1\over\displaystyle z-a_2-{\strut b_2\over\ddots}}}}.
\end{equation}
\end{rem}


\subsection{Operator models for the {\bf{t}}- and {\U}-transforms of measures}

 In this subsection we shall show that the operator deformations defined in Theorem \ref{thm_sa} (i) correspond to the \U- and {\bf t}-transforms of their distributions. 
More precisely, for a self-adjoint operator $A\in \B(\H)$ and a unit vector $u\in \H$ consider the vector state $\vp_u$. Then, to the operator $A$ there corresponds the distribution $\mu_A$ (given by \eqref{phiz}). On the other hand, we consider the operator deformation $A\mapsto \widetilde{A_{s,t}}$, which has the distribution $\mu_{\widetilde{A_{s,t}}}$ (also w.r.t. $\vp_u$). We shall show that in this case we get the equality  $\mu_{\widetilde{A_{s,t}}}=U^{p,q}(\mu_A)$, with $q=(1-s)(1-t)$, $p=1-s-t$. So the \U-transform of $\mu_A$ equals the distribution of the deformation $\widetilde{A_{s,t}}$. In particular we obtain the following diagram 
\[
\begin{CD}
A @>{\vp_u}>> \mu_A \\
@VVV @VV{U^{p,q}}V\\
\widetilde{A_{s,t}}@>{\vp_u}>> \mu_{\widetilde{A_{s,t}}} 
\end{CD} \qquad \Longrightarrow \quad \mu_{\widetilde{A_{s,t}}} =U^{p,q}(\mu_A).
\]

\begin{thm}\label{U-model}
Let $A$ be a self-adjoint, possibly unbounded operator, let $u\in\dom A$ be a normalized vector and let $\mu$ be the distribution of $A$ w.r.t.\ $\vp_u$, see \eqref{phiz}. If $(1-s)(1-t)\geq 0$, then the distribution of the rank two deformation 
$$ 
\widetilde{A_{s,t}}:= A- s\D{Au}{u}-t\D{u}{Au}
$$
 w.r.t.\ $\vp_u$ is the \U-transformation of the distribution $\mu$ with $q=(1-s)(1-t)$, $p=1-s-t$. 
\end{thm}
\begin{proof}
For  $s,t\in \R$ with $(1-s)(1-t)\geq 0$ the Theorem \ref{thm_sa} (i) guarantees that the Weyl function $\widetilde{Q^{s,t}_{u}} (z)$ of the operator $\widetilde{A_{s,t}}$ satisfies the equation 
\[
\frac{1}{\widetilde{Q^{s,t}_{u}} (z)} = \frac{(1- s)(1-t)}{Q_{u}(z)} + (s+t-st)z +stm. 
\]
On the other hand, the Weyl function $Q_u$ of the self-adjoint operator $A$ equals the Cauchy transform $G_{\mu}$ of its distribution $\mu$ (w.r.t. $\vp_u$). Hence
\[
\frac{1}{\widetilde{Q^{s,t}_{u}} (z)} = \frac{(1- s)(1-t)}{G_{\mu}(z)} + (s+t-st)z +stm = \frac{1}{G_{U^{p,q}(\mu)}(z)}, 
\]
where the second equality is the consequence of the Krystek-Yoshida construction of the \U-transform $U^{p,q}(\mu)$ of the measure $\mu$. Therefore $\widetilde{Q^{s,t}_{u}}=G_{U^{p,q}(\mu)}$, and, since the Cauchy transform determines the measure uniquely, the proof is finished.
\end{proof}


\begin{rem}\label{remark on s,t real}
We proved that the Weyl function  $\widetilde{Q^{s,t}_{u}} (z)$ of the rank two deformation $\widetilde{A_{s,t}}$ of a self-adjoint operator $A$ is the Cauchy transform of a probability measure. 
This is of particular interest in the case $s\neq t$, in which the deformation $\widetilde{A_{s,t}}$ is \textit{not self-adjoint}, but nevertheless its Weyl function w.r.t. $u$ is the Cauchy transform of the probability measure  $U^{p,q}(\mu)$.

\end{rem}

With the notation of the Theorem \ref{U-model} the first moment of $\mu$ equals $m=\<Au,u\>$. 
Since for measures with vanishing first moment both \U- and {\bf t}-transforms agree, we get the following.

\begin{cor}
Let $A$ be a self-adjoint, possibly unbounded operator, let $u\in\dom A$ be a normalized vector  and let $\mu$ be the distribution of $A$ w.r.t. $\vp_u$. Moreover, let $s,t\in \C$ be so that $\tau_{s,t}:=(1-s)(1-t)>0$ and assume that $m=\seq{Au,u}=0$. Then the distribution of the rank two deformation 
$$
\widetilde{A_{s,t}}:= A- s\D{Au}{u}-t\D{u}{Au}
$$
is  the $\mathbf{t}$-transform of the measure $\mu$ 
with the parameter $\tau_{s,t}$. 
\end{cor}

\subsection{Jacobi matrix models}
In this subsection we shall study our transformations acting on the Jacobi tridiagonal matrices. In fact, the Jacobi tridiagonal matrices are closely related to continuous fraction expansions  of the Cauchy transforms of probability measures, and to the associated orthogonal polynomials.

For simplicity we shall consider probability measures with compact supports on $\R$, which guarantees the existence of all moments. Then, given such measure $\mu$, there exists the family $\{P_n: n\geq 0 \}$ of polynomials orthogonal with respect to $\mu$ and normalized by $\|P_n\|^2=\int|P_n(x)|^2d\mu(x)=1$. If the Cauchy transform $G_{\mu}$ is of the form \eqref{Cauchy transform continued fraction}, then the polynomials satisfy the recurrence relation of the form 
\begin{equation}\label{recurence for orthogonal polynomials}
P_0(x)=1, \quad xP_{n}(x)=b_nP_{n+1}(x)+a_nP_{n}(x)+b_{n-1}P_{n-1}(x), \quad \text{for} \quad n\geq 0.
\end{equation}
The coefficients $b_n$ are positive and $a_n$ are real and bounded, (with the convention $b_{-1}=0$). Let us introduce the operator $J$, acting on $L^2(\R, d\mu(x))$ as multiplication by the variable $x$, then, in the orthonormal basis $e_n:=P_n$ ($n\in\N$), it has the tridiagonal Jacobi matrix 
\begin{equation}\label{Jacobi matrix}
J=\left[ \begin{array}{cccccccc}
a_0 & b_0 & 0 & 0 &0 &\dots  \\
b_0 & a_1 & b_1 & 0& 0& \dots \\
0 & b_1 & a_2 & b_2& 0& \dots \\
0 & 0 & b_2 &a_3 & b_3 & \dots\\
\vdots &\vdots &\vdots &\vdots &\vdots & \ddots
\end{array}\right].
\end{equation}
Moreover, the Weyl function of $J$ with respect to the vector $e_0$ and the Cauchy transform of $\mu$ coincide: 
\[
Q_{e_0}(z)=\<Je_0,e_0\>=\int \frac{\mu(dx)}{z-x}=G_\mu(z).
\] 

\subsubsection{Jacobi matrix model for the {\U}-transform}
For the Jacobi matrix \eqref{Jacobi matrix} consider the "antidiagonal" transformation  $J\mapsto \widetilde{J_{s,t}}=J-s(u\otimes Ju)-t(Ju\otimes u)$, for $u=e_0$, given in Theorem \ref{thm_sa} (i), then: 
\begin{equation}\label{Jacobi antidiagonal}
\begin{array}{rcl}
\widetilde{J_{s,t}}&:=&J-s(e_0\otimes Je_0)-t(Je_0\otimes e_0)\\
&=& J-(s+t)a_0(e_0\otimes e_0)-sb_0(e_0\otimes e_1)-tb_0(e_1\otimes e_0),
\end{array}
\end{equation}
and hence
\begin{equation}
\widetilde{J_{s,t}}=\left[ \begin{array}{cccccccc}
(1-s-t)a_0 & (1-s)b_0 & 0 & 0 &0 &\dots  \\
(1-t)b_0 & a_1 & b_1 & 0& 0& \dots \\
0 & b_1 & a_2 & b_2& 0& \dots \\
0 & 0 & b_2 &a_3 & b_3 & \dots\\
\vdots &\vdots &\vdots &\vdots &\vdots & \ddots
\end{array}\right].
\end{equation}
As we have seen in Theorem \ref{U-model}, the distribution of $\widetilde{J_{s,t}}$ (with respect to the vacuum state $\vp_0$ given by $e_0$) is the $U^{p,q}$-transform of the distribution of $J$ (with $p=1-s-t$ and $q=(1-s)(1-t)$): 
\[
\vp_0((z-\widetilde{J_{s,t}})^{-1})=\int\frac{U^{p,q}(\mu)(dx)}{z-x}.
\]

In the particular case if $a_0=\<Je_0,e_0\>=0$ we get that the distribution $U^{p,q}(\mu)$ of the operator $\widetilde{J_{s,t}}$, (non self-adjoint if $s\neq t$), is the {\bf t}-transform of the distribution $\mu$ of $J$. 

\subsection{{\bf W}-transform of measures and related operator model}

The Theorem \ref{thm_sa} (ii) provides a possibility of defining another  transformation of  probability measures via the following lemma.

\begin{lemma}
Let $\mu$ be a Borel probability measure on $\Real$ with the first moment $m$ finite. 
Then for each $s,t\in\Real$ the function
$$
F(z)=s + \frac{1+t(z^2G_{\mu}(z) -m-z)}{(1-tm+tz)G_{\mu}(z)-t}
$$
is a reciprocal of a Cauchy transform of a probability measure.
\end{lemma}
 
 \begin{proof}
 It is well known that there exists a self-adjoint operator $A$ in a Hilbert space and a unit vector $u\in\dom A$ such that the distribution of $A$ with respect to $u$ equals $\mu$. By Theorem \ref{thm_sa}(ii) $F(z)$ is the reciprocal of the distribution of $A-s \D uu -t \D{Au}{Au}$ with respect to $u$.
 \end{proof}

\begin{defi}
If $\mu$ is a Borel probability measure on $\Real$ with the first moment $m$ finite and $s,t\in\Real$, then by the \WT-transform of $\mu$ we define the unique Borel probability measure $W^{s,t}(\mu)$ on $\Real$ for which  the Cauchy transform $G_{W^{s,t}(\mu)}$ satisfies the equation 
\begin{equation}\label{W-transform}
\frac{1}{G_{W^{s,t}(\mu)}(z)} = s + \frac{1+t(z^2G_{\mu}(z) -m-z)}{(1-tm+tz)G_{\mu}(z)-t}.
\end{equation}
\end{defi}

We present below two simple examples comparing the two deformations: $\mu \mapsto U^{p,q}(\mu)$ for $p=1-s-t$, $q=(1-s)(1-t)$, and $\mu \mapsto W^{s,t}(\mu)$. One more example will be treated in Section \ref{sec_meixner}. 

\begin{ex} \label{ex_delta}
Both \U- and \WT-tranform preserve the class of atomic measures. 
Indeed, for $\mu=\delta_a$ we have $m=a$ and $\displaystyle G_{a}(z)=\frac{1}{z-a}$. Hence 
$$
\widetilde{U^{p,q}}(\delta_a)=\delta_{pa}\ \text{ and }\ \displaystyle W^{s,t}(\delta_a) = \delta_{a-s-ta^2}.
$$
The latter formula follows directly from \eqref{W-transform}, which simplifies to
\[
G_{W^{s,t}(\delta_a)}(z)=\frac{1}{z-(a-s-ta^2)}.
\]
\end{ex}

\begin{ex} \label{ex_bernoulii}
For the Bernoulli law $\mu:=\frac{1}{2}(\delta_{-1}+\delta_{1})$ we have $m=0$, so the \U-transform reduces to the {\bf t}-transform. General formulas for the {\bf t}-transform of two-point measure has been given in \cite[Example 3.5]{wojakowski07}, and applied to the Bernoulli law give
$$\widetilde{U^{p,q}} (\mu)= \frac12\big(\delta_{-\sqrt{q}}+\delta_{\sqrt{q}}\big).$$

On the other hand, the Cauchy transform is $ G_{\mu}(z)=\frac{1}{2}\left(\frac{1}{z-1}+\frac{1}{z+1} \right)$, hence the \WT-transform is given by 
\[
G_{W^{s,t}(\mu)}(z)= \frac{z+t}{z^2+(s+t)z+st-1}=\frac{A_{s,t}}{z-x}+\frac{B_{s,t}}{z-y},
\]
with 
\[
A_{s,t}=\frac{x+t}{x-y}, \quad B_{s,t}=\frac{y+t}{y-x},
\]
where $x\neq y$ are the two real solutions of the quadratic equation $z^2+(s+t)z+st-1=0$ (which has the positive discriminant $\Delta=(s-t)^2+4$):
\[
x:=x_{s,t}=\frac{-(s+t)+\sqrt{(s-t)^2+4}}{2}, \quad y:=y_{s,t}=\frac{-(s+t)-\sqrt{(s-t)^2+4}}{2}.
\]
Therefore we get $W^{s,t}(\mu) = A_{s,t}\delta_{x_{s,t}}+B_{s,t}\delta_{y_{s,t}}$. In the particular case $s=t$ we obtain $x=1-s$, $y=-1-s$ and $A_{s,t}=B_{s,t}=\frac{1}{2}$, so that the atoms are shifted by $s$ and then $W^{s}(\mu) = \frac{1}{2}(\delta_{-1-s}+\delta_{1-s})$. 
\end{ex}

\subsubsection{Jacobi matrix model for the {\WT}-transform}
For the Jacobi matrix \eqref{Jacobi matrix} consider the deformation  $J\mapsto \widehat{J_{s,t}}=J-s(u\otimes u)-t(Ju\otimes Ju)$, for $u=e_0$, given in Theorem \ref{thm_sa} (ii), then: 
\begin{equation}\label{Jacobi diagonal}
\begin{array}{rcl}
\widehat{J_{s,t}}&:=&J-s(e_0\otimes e_0)-t(Je_0\otimes Je_0)\\
&=& J-(s+ta_0^2)(e_0\otimes e_0)-tb_0^2(e_1\otimes e_1)-ta_0b_0(e_0\otimes e_1+e_1\otimes e_0).
\end{array}
\end{equation}
Using $a_0=m$ the deformed matrix can be written as
\begin{eqnarray*}
\widehat{J_{s,t}}&=& \left[ \begin{array}{ccccccc}
a_0(1 - ta_0) -s& b_0(1-t a_0) & 0 & 0 &  \\
b_0(1- t a_0) & a_1 -t b_0^2 & b_1 & 0& \\
0 & b_1 & a_2 &  \ddots&\\
0&0 & \ddots & \ddots & 
\end{array}\right]\\
&=& \left[ \begin{array}{ccccccc}
m(1 - tm) -s& b_0(1-t m) & 0 & 0 &  \\
b_0(1- t m) & a_1 -t b_0^2 & b_1 & 0& \\
0 & b_1 & a_2 &  \ddots&\\
0&0 & \ddots & \ddots & 
\end{array}\right].
\end{eqnarray*}
For the vector state $\vp_0$, given by $u=e_0$, we get 
\[
\vp_0((z-\widehat{J_{s,t}})^{-1})= \langle (z-\widehat{J_{s,t}})^{-1}e_0, e_0 \rangle = \int \frac{W_{s,t}(\mu)(dx)}{z-x}, 
\]
so the deformation $J\mapsto \widehat{J_{s,t}}$ gives the Jacobi matrix model for the {\WT}-transform.

\subsection{Transformations of the free Meixner class} \label{sec_meixner}

We end this section with another new result, which is the decription of the behaviour of the \textit{free Meixner class } under the \U-transform and describe the \U- and \WT-transforms of the Wigner law.

In classical probability the Meixner laws form the class of probability measures, whose orthogonal polynomials $p_n(x)$ are the solutions of the equation of the form  $\displaystyle \sum_{n=0}^{\infty}p_n(x)y^n=A(y)e^{xB(y)}$, for given analytic functions $A, B$. It contains the classical laws:  Gaussian, Poisson, gamma, Pascal, binomial and hyperbolic secant (see \cite{andrews-askey} and \cite{bozejko_bryc06}). 

The \textit{free} Meixner laws are analogues of the above in free probability (developed by Voiculescu in mid 80-thies of the last century, c.f. \cite{VDN}), and contain the free Gaussian (Wigner semicircle) law, free Poisson (Marchenko-Pastur) law, free Gamma and free binomial law. Their orthogonal polynomials satisfy the recursion with constant coefficients. The free Meixner class has been described by Saitoh and Yoshida \cite{saitoh_yoshida} as the class of probability measures $\mu_{\bf u}$ on $\R$ depending on four parameters ${\bf u}:=(\gamma, a, b, c)\in \R^4$ with $b, c \geq 0$, 
which orthogonalize the family of polynomials $P_n(x):=P^{\bf u}_n(x)$ given by the following recurence:
\begin{equation}
P_0(x):=c, \ P_{1}(x):=x-\gamma, \ \text{and}\  P_{n+1}(x)=(x-a)P_{n}(x)-bP_{n-1}(x), \ \text{for}\ n\geq 1.
\end{equation}
As shown in \cite{saitoh_yoshida} the measure $\mu_{\bf u}$ has the absolutely continuous part supported on the interval $[a-2\sqrt{b}, a+2\sqrt{b}]$, with the density function 
\begin{equation}\label{Meixner class}
\mu_{\bf u}(dx):=\frac{c\sqrt{4b-(x-a)^2}}{2\pi f(x)}=\frac{c\sqrt{4b-y^2}}{2\pi g(y)}=\mu_{\bf u}(dy), 
\end{equation}
where $f(x):=(1-c)(x-a)^2+(c-2)(\gamma-a)(x-a)+(\gamma-a)^2+bc^2:=g(y)$, with $y:=x-a$. There is no singular part and the apperance of atoms is ruled by the following properties: 
\begin{enumerate}
\item if $f$ has two real roots $x_1\neq x_2$, i.e. the discriminant $\displaystyle \Delta_g:=c^2[(\gamma-a)^2-4b(1-c)]>0$ is positive, then the atoms are in $x_1$ and  $x_2$;
\item if $c=1$ and $\gamma \neq a$, i.e. $f$ is a linear function with the root $\displaystyle x_0=\gamma+\frac{bc^2}{\gamma -a}$, then the atom is in $x_0$. 
\end{enumerate}
Otherwise, there is no atoms, and the measure is absolutely continuous. 


The Cauchy transform of such measure $\mu_{\bf u}$ has the following  continued fraction expansion:
\begin{equation}\label{Meixner class continued fraction}
G_{\mu_{\bf u}}(z) = {1\over\displaystyle z-\gamma- {\strut bc\over\displaystyle
z-a- {\strut b\over\displaystyle z-a-{\strut b\over\ddots}}}}
\end{equation}
\begin{prop}
The free Meixner class is invariant under the {\U}-transform $\mu_{\bf u}\mapsto \widetilde{U^{s,t}}(\mu_{{\bf u}})$ for $(1-s)(1-t)>0$. In particular, 
$$
(\gamma, a,b,c)={\bf u}\longmapsto {\bf u}_{s,t}:=\big((1-s-t)\gamma,\ a,\ b,\ c(1-s)(1-t)\big).
$$
\end{prop}

\begin{proof}
The transformation $\mu_{\bf u}\mapsto \widetilde{U^{s,t}}(\mu_{{\bf u}})$, viewed through  the Cauchy transform, has the following continued fraction expansion, see \eqref{U-transform of continued fraction},
\begin{equation}
G_{\widetilde{U^{s,t}}(\mu_{{\bf u}})}(z) = {1\over\displaystyle z-(1-s-t)\gamma- {\strut bc(1-s)(1-t)\over\displaystyle
z-a- {\strut b\over\displaystyle z-a-{\strut b\over\ddots}}}}.
\end{equation}
so that it maps $\gamma \longmapsto (1-s-t)\gamma$, $c\longmapsto c(1-s)(1-t)$.
Therefore,  for $(1-s)(1-t)>0$, the transformed measure $\widetilde{U^{s,t}}(\mu_{{\bf u}})$ is again in the free Meixner class. 
\end{proof}

Special cases of the free Meixner laws are the Wigner law for ${\bf u}=(0,0,1,1)$ and the Marchenko-Pastur law (the free Poisson distribution with jump size $\alpha$ and rate $\lambda$) for ${\bf u}=(\alpha \lambda, \alpha(1+\lambda), \alpha^2\lambda,1)$. We shall concentrate on describing the {\U} and {\WT}-transforms of the first of them.

\begin{ex}[Wigner  semicircle law] \label{ex_wigner}
The Wigner semi-circle law appears in the central limit theorem for free probability, and it is the absolutely continuous probability measure with density 
\[
\sigma(dx)=\frac{1}{2\pi}\sqrt{4-x^2} \quad \text{for}\quad -2\leq x \leq 2.
\]
Its Cauchy transform is 
$$
G_\sigma (z)= \frac{z-z\sqrt{1-\frac{4}{z^2}}}{2}.
$$
The Cauchy transform is well-defined for $z\in \C^+$, and the branch of the square root is chosen so that $G_\sigma(z)\in \C^-$. The function $ G_\sigma(z)$ extends analytically to a neighborhood of the part of the real axis $(-\infty,-2)\cup (2,+\infty)$. 

\begin{itemize}

\item[{\U}.] First we describe the {\U}-transform of the Wigner semicircle law. 
Since the Wigner measure $\sigma$ is in the free Meixner class with parameters $\gamma=a=0$ and $b=c=1$,  putting these into \eqref{Meixner class} we see, that the measure $\widetilde{U^{s,t}}(\sigma)$ is again in the free Meixner class with density
\[
\widetilde{U^{s,t}}(\sigma)(dx)=\frac{\tau_{s,t}\sqrt{4-x^2}}{(1-\tau_{s,t})x^2+\tau_{s,t}^2}, \quad \text{where}\quad \tau_{s,t}:=(1-s)(1-t).
\]
The measure can have either two atoms or none, since $0=\gamma \neq c=1$. Two atoms can appear if and only if the denominator $f_{s,t}(x):=(1-\tau_{s,t})x^2+\tau_{s,t}^2$ has two real roots, which is possible if and only if $\tau_{s,t}>1$. Then the atoms are in 
\[
x_+=\frac{\tau_{s,t}}{\sqrt{\tau_{s,t}-1}}, \quad x_-=\frac{-\tau_{s,t}}{\sqrt{\tau_{s,t}-1}}.
\]
Here the transition line is the hyperbola $\tau_{s,t}=(1-s)(1-t)=1$, on which the transformation is trivial: $\displaystyle \widetilde{U^{s,t}}(\sigma)=\sigma$. Examples of numerically obtained plots of densities of $\widetilde{U^{s,t}}(\sigma)(dx)$ can be seen in Figure \ref{densities}.

\medskip

\item[{\WT}.] Let us now consider the \WT-transform of $\sigma$. The distribution of $W^{s,t}(\sigma)$  satisfies the defining equation 
\begin{equation*}
\frac{1}{G_{W^{s,t}(\sigma)}(z)} = s+\frac{1+tz^2G_\sigma(z)-tz}{(1+tz)G_\sigma(z)-t}=s+z-\frac{1}{t+\frac{1}{z}+\frac{1}{z^2G_\sigma(z)-z}}.
\end{equation*}
Using $G_\sigma(z)=\frac{1}{z-G_\sigma(z)}$ this can be simplified to
\begin{equation*}
\frac{1}{G_{W^{s,t}(\sigma)}(z)} = s+z-\frac{1}{t+\frac{1}{G_\sigma(z)}}\ \Leftrightarrow \ G_{W^{s,t}(\sigma)}(z) = \frac{1}{z+s-\frac{1}{z+t-G_\sigma(z)}}.
\end{equation*}
 It is instructive to compare the continued fraction expansion of the second equation, namely 
\begin{equation*}
G_{W^{s,t}(\sigma)}(z) = {1\over\displaystyle z+s- {\strut
1\over\displaystyle z+t- {\strut 1\over\displaystyle z-
{\strut 1\over\displaystyle z- {\strut 1\over\ddots}}}}}
\quad \mbox{with}\quad 
G_{\sigma}(z) = {1\over\displaystyle z- {\strut
1\over\displaystyle z- {\strut 1\over\displaystyle z-
{\strut 1\over\displaystyle z- {\strut 1\over\ddots}}}}}.
\end{equation*}
As one can see our transformation acts in a specific way on two levels of the continued fraction in the Wigner semicircular case, however it is quite different from the two-levels $\mathbf{t}$-transformation studied by Wojakowski in \cite{wojakowski10}.

The Cauchy transform can be expressed directly as 
\begin{equation*}
G_{W^{s,t}(\sigma)}(z)= \frac{4tz^2+(4t^2+4st+2)z+4(st^2+s-t)-2z\sqrt{1-\frac{4}{z^2}}}{4[tz^3+(t^2+2st)z^2+(s^2t+2st^2+s-2t)z+s^2(t^2+1)-2st+1]}.
\end{equation*}
Using the Stieltjes inversion formula one gets the density of the \WT-transformation of the Wigner law (for $|x|\leq 2$):
\[
W^{s,t}(\sigma)(dx)=\frac{1}{2\pi}\cdot\frac{\sqrt{4-x^2}}{tx^3+(t^2+2st)x^2+ (s^2t+2st^2+s-2t)x + (st-1)^2+s^2}. 
\]
This defines a family of absolutely continuous  measures on $[-2,2]$. As  can be seen in the picture obtained by numerical simulations in Figure \ref{densities} these measures can have atoms outside $[-2,2]$.   

Let us observe that the measure $W^{s,t}(\sigma)$ is not of the form \eqref{Meixner class}, hence the free Meixner class is not invariant under the \WT-transform. 
\end{itemize}
\end{ex}



\subsection{$\widetilde{U^{s}}$-transformation of the free Meixner laws}

We now focus on the case $s=t$, writting $\widetilde{U^{s}}:=\widetilde{U^{s,s}}$. 
We shall describe the phase transition for the deformation $\mu_{{\bf u}}\mapsto \widetilde{U^{s}}(\mu_{{\bf u}})$, i.e. the problem when the number of atoms changes. In this case we get the transformation of the parameter's vectors
$$
(\gamma, a,b,c)={\bf u}\longmapsto {\bf u}_{s}:=((1-2s)\gamma,\ a,\ b,\ c(1-s)^2),
$$
hence $\widetilde{U^{s}}(\mu_{{\bf u}}) = \mu_{{\bf u}_s}$.
Let us notice that the inverse transform $(\widetilde{U^s)}^{-1}$ is defined by 
$$
(\widetilde{U^s)}^{-1}:(\gamma, a,b,c)={\bf u}\longmapsto \widetilde{{\bf u}_s}:=\left(\frac{\gamma}{1-2s},\ a,\ b,\ \frac{c}{(1-s)^2}\right)
$$ whenever $s\neq 1$ and $s\neq \frac{1}{2}$, which we shall assume in the sequel. Moreover, we shall consider only the case $c\neq 0$ (for $c=0$ we get $\mu=\delta_{\gamma}$). The invertibility of the $\widetilde{U^s}$-transform simplifies the description of the phase transitions in what follows.


The density function can be written as 
\[
\mu_{{\bf u}_{s}}(dx)=\frac{c(1-s)^2\sqrt{4b-(x-a)^2}}{2\pi f_{s}(x)}. 
\]
The transformed measure $\mu_{{\bf u}_s}$ has two atoms if  and only if
$$\displaystyle \Delta_{g_s}:=c^2(1-s)^4[((1-2s)\gamma-a)^2-4b(1-c(1-s)^2)]>0
$$
i.e. if the discriminant $\Delta_{g_s}$ is positive, which, for $s\neq 1$, is equivalent to the inequality 
$\displaystyle ((1-2s)\gamma-a)^2 > 4b(1-c(1-s)^2)$. On the other hand one atom can appear if and only if $c(1-s)^2=1$ and $(1-2s)\gamma \neq a$. The case $s=1$ is degenerated, since then 
$\mu_{{\bf u}_1}=\delta_{\gamma}$ is the atomic measure with one atom.

Now we consider the phase transition produced by the \U-transform. We first describe the case when $\mu$ has one atom and $\mu_{{\bf u}_s}$ has two atoms. 


\begin{prop}
If the free Meixner law $\mu_{\bf u}$, with ${\bf u}=(\gamma, a,b,c)$ and $\gamma\neq 0$ has one atom, then there exists an uncountable range of parameter $s\neq 0$ for which $\mu_{{\bf u}_s}$ has two atoms. This range depends on the position of the point $P_{\bf u}:=(\frac{a}{\gamma},  \frac{b}{\gamma^2})$ with respect to the elipse $E:=\{(x,y)\in \R^2: x^2+4x+4y^2-5y-2xy=0\}$ as follows: 
\begin{enumerate}
\item if $P_{\bf u}$ is inside $E$, then $s\in \R\setminus\{0\}$; 
\item if $P_{\bf u}\in E$ then $s\in \R\setminus \{0, s_0\}$, where $s_0$ is the double root of the quadratic polynomial $4(b+\gamma^2)s^2+4(a\gamma-2b+\gamma^2)s+(\gamma^2-2a\gamma)$;
\item  if $P_{\bf u}$ is outside $E$,  then $s<s_1$ or $s>s_2$, where $s_1, s_2$ are the (different) roots of the quadratic polynomial $4(b+\gamma^2)s^2+4(a\gamma-2b+\gamma^2)s+(\gamma^2-2a\gamma)$.
\end{enumerate}
\end{prop}

\begin{proof}
Since $\mu_{\bf u}$ has one atom, hence $c=1$ and $\gamma \neq a$. Then 
$\mu_{{\bf u}_s}$ has two atoms for these $s$ for which 
\[
\varphi(s):=4(b+\gamma^2)s^2+4(a\gamma-2b+\gamma^2)s+(\gamma^2-2a\gamma)>0. 
\]
There are three possible cases, depending on the sign of the discriminant $\Delta(\varphi)=4a\gamma^3+(a^2-5b)\gamma^2 -2ab\gamma+4b^2$.   

Putting $a=x \gamma$ and $b=y\gamma^2$, (and assuming $\gamma\neq 0$), the condition $\Delta(\varphi)=0$ is equivalent to $x^2+4x+4y^2-5y-2xy=0$, i.e. the point $P_{\bf u}$ is on the ellipse $E$. In a similar manner we get the two other cases. 
\end{proof}
\begin{rem}
If $\gamma=0$ then $\vp(s)=4bs^2-8bs>0$ gives $s<0$ or $s>2$. Hence any Meixner class measure $\mu$ with one atom and with vanishing first moment $\gamma=0$ is transformed into a Meixner class measure $U^s(\mu)$ with two atoms. 
\end{rem}


The next phase transition is from no atoms in $\mu$ to one atom in $\mu_{{\bf u}_s}$ ($s\neq 1$). In what follows we shall use the notation $\alpha:=\frac{a}{\gamma}$ and  $\beta:=\frac{b}{\gamma^2}$. 

\begin{prop}
If the free Meixner law $\mu_{\bf u}$, with ${\bf u}=(\gamma, a,b,c)$ and $\gamma\neq 0$, has no atoms, then the \U-transform $U^s(\mu_{\bf u}):=\mu_{{\bf u}_s}$ has one atom for the following ranges of the parameter $s\in \R$: 
\begin{enumerate}
\item[(1a)] if $c=1$ and $\gamma=a$ then $s=2$;
\item[(1b)] if $c\neq 1$ and $(\gamma-a)^2\leq4b(1-c)$, then for  $r:=\frac{1-\alpha}{2\sqrt{\beta}}$
\begin{equation}\label{0 to 1 atom}
s < 1- \left(1-r^2\right)^{-1/2} \ \text{or} \ s > 1+ \left(1-r^2\right)^{-1/2}.
\end{equation}
\end{enumerate}
\end{prop}

\begin{proof}
\begin{enumerate}
\item[(1a)] 
If $c=1$ and $\gamma=a$, then $f(x)\equiv b$ is a constant function, and $\mu_{\bf u}$ is  related to the Wigner law (Example \ref{ex_wigner}): $\displaystyle \sigma(dw)=\frac{1}{\pi}\sqrt{1-w^2}$ by the transformation $\displaystyle w=\frac{x-a}{2\sqrt{b}}$. Then $\mu_{{\bf u}_s}$ has one atom if and only if $qc=1$ and $p\gamma\neq a$, for $q=(1-s)^2$ and $p=1-2s$. Thus $p\neq 1$ implies $s\neq 0$ and thus $q=1$ must imply $s=2$.  


\item[(1b)]
If $c\neq 1$ and $(\gamma-a)^2\leq4b(1-c)$ (i.e.$\Delta_g \leq 0$), then $\left(\frac{1-\alpha}{2\sqrt{\beta}} \right)^2+c\leq 1$
Then $\mu_{{\bf u}_s}$ can have exactly one atom if and only if $qc=1$ (i.e. $s\neq 1$ is so that $c= \frac{1}{(1-s)^2}$) and $p\gamma\neq a$ (i.e. $\alpha \neq 1-2s$) and $\left(\frac{1-\alpha}{2\sqrt{\beta}} \right)^2+\left(\frac{1}{s-1} \right)^2\leq 1$, which is equivalent to \eqref{0 to 1 atom}.
\end{enumerate}
\end{proof}

The last case of the phase transition is from no atoms in $\mu$ to two atoms in $\mu_{{\bf u}_s}$.

\begin{prop}
If the free Meixner law $\mu_{\bf u}$, with ${\bf u}=(\gamma, a,b,c)$ and $\gamma\neq 0$, has no atoms, then the \U-transform $U^s(\mu_{\bf u}):=\mu_{{\bf u}_s}$ has two atoms for the following ranges of the parameter $s\in \R$: 
\begin{enumerate}
\item[(2a)] if $c=1$ and $\gamma=a$ then $s\neq 2$ and $s>\frac{2b}{a^2+b}$;
\item[(2b)] if $c\neq 1$ and $(\gamma-a)^2 < 4b(1-c)$, then $s<s_1$ or $s>s_2$, where $s_1<s_2$ are the solutions of $(c\beta+1)s^2-2s(c\sqrt{\beta}+r)+d\beta=0$.
\item[(2c)] if $c\neq 1$ and $(\gamma-a)^2 = 4b(1-c)$ and $cb+\gamma^2-a\gamma=0$, then $s\in \R\setminus \{0\}$;
\item[(2d)] if $c\neq 1$ and $(\gamma-a)^2 = 4b(1-c)$ and $cb+\gamma^2-a\gamma\neq 0$, then $s<\min\{0, s_2\}$ or  $s>\max\{0, s_2\}$, where  $s_2=2-\frac{\gamma^2+a\gamma}{\gamma^2+bc}$.
\end{enumerate}
\end{prop}

\begin{proof}
The transformed measure $\mu_{{\bf u}_s}$ has two atoms if and only if the discriminant $\Delta_{g_s}=c^2(1-s)^4[((1-2s)\gamma-a)^2-4b(1-c(1-s)^2)]$ is positive, or equivalently, $((1-2s)\gamma-a)^2>4b(1-c(1-s)^2)$.

\begin{enumerate}
\item[(2a)]
If $c=1$ and $\gamma=a$ then $\mu_{{\bf u}_s}$ has two atoms if and only if $qc\neq 1$ (i.e. $0\neq s\neq 2$) and $4(1-s)^4(a^2s+b(s-2))>0$, which is equivalent to 
$s>\frac{2b}{a^2+b}$. Since $2>\frac{2b}{a^2+b}$, the case $s=2$ must be excluded from the range of $s$ and this way we get the conclusion of (2a). 
\item[(2b)]
In this case we have $c\neq 1$, $(\gamma-a)^2 < 4b(1-c)$ and $((1-2s)\gamma-a)^2 > 4b(1-c(1-s)^2)$. 
Thus the phase transition (no atoms in $\mu_{\bf u}$ and 2 atoms in $\mu_{{\bf u}_s}$) would be if simultanously 
\[
\left(\frac{1-\alpha}{2\sqrt{\beta}} \right)^2+c< 1, \quad \text{and}\quad \left(\frac{1-\alpha}{2\sqrt{\beta}}-\frac{s}{\sqrt{\beta}}\right)^2+c\left(1-s\right)^2>1, 
\] (with $\alpha=\frac{a}{\gamma}, \beta=\frac{b}{\gamma^2}$). 
Observe that this does not happen for $s=0$. 
Putting $r=\frac{1-\alpha}{2\sqrt{\beta}}$ and $d=r^2+c-1$ the above becomes equivalent to $d<0$ and
\[
\psi(s):=(c\beta+1)s^2-2s\sqrt{\beta}(c\sqrt{\beta}+r)+d\beta>0.
\]
Since both $c$ and $\beta$ are positive and $c\beta+1=\frac{cb}{\gamma^2}+1\geq 1$ and $d<0$, the discriminant $\Delta_{\psi}$ of the quadratic polynomial $\psi$ is positive and $\psi(0)<0$. Hence $\psi$ has two real roots $s_1<0< s_2$, and thus $\psi(s)>0$ if and only if $s<s_1$ or $s>s_2$. 

\item[(2c)] If $(\gamma-a)^2 = 4b(1-c)$ and $cb+\gamma^2-a\gamma=0$, then $d=0$ and $c\sqrt{\beta}+r=0$, hence $s_1=0$ is a double root of $\psi(s):=(c\beta+1)s^2-2s\sqrt{\beta}(c\sqrt{\beta}+r)+d\beta=(c\beta+1)s^2$. Thus $\psi(s)>0$ if and only if $s\neq 0$. 

\item[(2d)] If $(\gamma-a)^2 = 4b(1-c)$ and $cb+\gamma^2-a\gamma\neq 0$, then $d=0$ and  $s_1=0$ is a root of $\psi(s):=(c\beta+1)s^2-2s\sqrt{\beta}(c\sqrt{\beta}+r)+d\beta=(c\beta+1)s^2-2s\sqrt{\beta}(c\sqrt{\beta}+r)$. The second root is $s_2=2-\frac{\gamma^2+a\gamma}{cb+\gamma^2}$. Observe that $s_2>0$ if and only if $\gamma^2-a\gamma+2bc>0$ and using $(\gamma-a)^2=4b(1-c)$ this can be written equivalently as $2b(2-c)>a^2$. 

\end{enumerate}

\end{proof}





\begin{thebibliography}{Wor87a}

\bibitem{Aaronson} J. Aaronson, An introduction to infinite ergodic theory, Mathematical Surveys and Monographs, 50, American Mathematical Society, Providence, RI, 1997.

\bibitem{andrews-askey}
G. Andrews, R. Askey, Richard 
\newblock Classical orthogonal polynomials. Orthogonal polynomials and applications \newblock {(Bar-le-Duc, 1984), 36Â–62, Lecture Notes in Math.}, 1171, Springer, Berlin, 1985.

\bibitem{Batzkeetal} L. Batzke, C. Mehl, A.C.M. Ran, L. Rodman,
Generic rank-k perturbations of structured matrices,
Technical Report No. 1078, DFG Research Center Matheon, Berlin, 2015.

\bibitem{Berres1} D. K. Basson, S. Berres, R. B\"urger,
On models of polydisperse sedimentation with particle-size-specific hindered-settling factors, Applied Mathematical Modelling 33 (2009) 1815--1835.



\bibitem{Berres2} S. Berres, T. Voitovich, On the spectrum of rank two modification of a diagonal matrix for linearized fluxes modelling polydisperse sedimentation, Proceedings of Symposia in Applied Mathematics, 66.2(2009), 409--418.

\bibitem
{bozejko_bryc06}
M. Bo{\.z}ejko, W. Bryc, 
\newblock On a class of free L\'evy laws related to a regression problem
\newblock{\em J. Funct. Anal.} 236 (2006), 59-77.


\bibitem
{bozejko_wysoczanski98}
M. Bo\.{z}ejko, J. Wysocza\'{n}ski,
\newblock New examples of convolutions and non-commutative central limit
theorems.
\newblock {\em Banach Center Publ.} 43 (1998), 95-103.

\bibitem
{bozejko_wysoczanski01}
M. Bo\.{z}ejko, J. Wysocza\'{n}ski,
\newblock Remarks on $t$-transformations of measures and convolutions.
\newblock {\em Ann. Inst. H. Poincar\'e Probab. Statist.} 37 (2001), no. 6,
737--761.


\bibitem{ByHeM98} R.~Byers, C.~He, and V.~Mehrmann. Where is the nearest non-regular pencil?
\emph{Linear Algebra Appl.}, 285 (1998), 81--105.

\bibitem{DHS1}
V.~Derkach, S.~Hassi, and H.S.V.~de~Snoo,
``Operator models associated with Kac subclasses of
generalized Nevanlinna functions'',
Methods of Functional Analysis and Topology, 5 (1999), 65--87.

\bibitem{DHS3}
V.A.~Derkach, S.~Hassi, and H.S.V.~de~Snoo,
``Rank one perturbations in a Pontryagin space with one negative square'',
J. Funct. Anal., 188 (2002), 317--349.

\bibitem{DPW16} M. Derevyagin, L. Perotti, M. Wojtylak, Truncations of a class of pseudo-Hermitian operators, arxiv: 1503.04314v2. 


\bibitem{DoMoTe08} F. De Ter\'an, F.M. Dopico, J. Moro,   First order spectral perturbation theory of square singular matrix pencils,
\textit{Linear Algebra and its Applications}, 429 (2008) 548--576.

\bibitem{DoTe10} F. De Ter\'an, F. Dopico, First order spectral perturbation theory of square singular matrix polynomials,
\textit{Linear Algebra and its Applications}, 432 (2010) 892--910.

\bibitem{DM} R. Donat, P. Mulet, A secular equation for the Jacobian matrix of certain multi-species kinematic flow models, \textit{Num. Meth. for Partial Differrential Equations}, 26(2008), 159 -- 175.  

\bibitem
{Freund10}{R. W. Freund, Recent Advances in Structure-Preserving Model
Order Reduction}

\bibitem{Greenstein} D.S. Greenstein, On Analytic Continuation of Functions which Map the Upper Half Plane into Itself, \textit{J. Math. Ann. Appl.}, 1 (1960), 355--362.

\bibitem
{hasebe12}
T. Hasebe,
\newblock Analytic continuations of Fourier and Stieltjes transforms and
generalized moments of probability measures.
\newblock {\em J. Theoret. Probab.} 25 (2012), no. 3, 756--770.


\bibitem{kato}
T. Kato, {\em Perturbation Theory for Linear Operators}, Springer, New York 1966.

\bibitem{KrYos2004} A. Krystek, H. Yoshida, Generalized t-transformations of probability measures and deformed convolutions. Probab. Math. Statist. 24 (2004), no. 1, Acta Univ. Wratislav. No. 2646, 97--119.

\bibitem{MMRR1}
C. Mehl, V. Mehrmann, A.C.M. Ran, L. Rodman, Eigenvalue perturbation theory of classes of
structured matrices under generic structured rank one perturbations, \emph{Lin. Alg. Appl.}
425 (2011), 687-716.

\bibitem{MMRR2} C. Mehl, V. Mehrmann, A.C.M. Ran and L. Rodman, Perturbation theory of self-adjoint matrices
 and sign characteristics under generic structured rank one perturbations, \emph{Lin. Alg.
Appl.} (2010), doi:10.1016/j.laa.2010.04.008.


\bibitem{MMRR3} C. Mehl, V. Mehrmann, A.C.M. Ran, L. Rodman,
Jordan forms of real and complex matrices under rank one perturbations.
Submitted for publication.


\bibitem{MMW}{ 	C. Mehl, V. Mehrmann, M. Wojtylak, On the distance to singularity via low rank perturbations, {\it Operators and Matrices}, 9 (2015), 733--772.}

\bibitem{Mlotkowski Sakuma BCP 2012} W. M\l{}otkowski, N. Sakuma, Symmetrization of probability measures, pushforward of order 2 and the Boolean convolution. Noncommutative harmonic analysis with applications to probability III, 271--276, \emph{Banach Center Publ.}, 96, Polish Acad. Sci. Inst. Math., Warsaw, 2012.


\bibitem{MBO}
J. Moro, J.V. Burke,  M.L. Overton, On the Lidskii--Vishik--Lyusternik perturbation theory for
eigenvalues of matrices with arbitrary Jordan structure, \emph{SIAM J. Matrix Anal. Appl.} 18
(1997), 793-817.



\bibitem{nelson}{E. Nelson, Analytic vectors, {\it Ann. of Math.,}  70 (1959),
572--615.}


\bibitem{RW12} A.C.M. Ran, M. Wojtylak, Eigenvalues of rank one perturbations of unstructured matrices, \textit{
Linear Algebra Appl.} 437 (2012), no. 2, 589?600.

\bibitem
{saitoh_yoshida}
N. Saitoh, H. Yoshida, 
\newblock The infinite divisibility and orthogonal polynomials with a constant recursion formula in free probability theory. 
\newblock {\em Probab. Math. Stat.} vol. 21, fasc. 1 (2001), 159-170.

\bibitem{Sa1}
S.V. Savchenko,  On a generic change in the spectral properties under perturbation by an
operator of rank one, [Russian] \textit{ Mat. Zametki}  74  (2003),  590--602; [English]
\textit{Math. Notes}  74  (2003),  557-568.

\bibitem{Sa0}
S.V. Savchenko, On the Perron roots of principal submatrices of co--order one of irreducible nonnnegative matrices, \textit{Linear Algebra and its Applications} 361 (2003), 257-277.



\bibitem{Sa2}
S.V. Savchenko,  On the change in the spectral properties of a matrix under a perturbation of
a sufficiently low rank, [Russian] \textit{ Funkcional. Anal. i Prilozhen.}  38 (2004),
85--88; [English]  \textit{Funct. Anal. Appl.}  38  (2004),  69-71.

\bibitem{SWW1}
H.S.V.~de~Snoo, H.~Winkler, and M.~Wojtylak,
"Zeros and poles of nonpositive type of  Nevanlinna
functions with one negative square",
J. Math. Ann. Appl., 382 (2011), 399--417.



\bibitem{SWW2}{ H.S.V. de Snoo, H. Winkler, M. Wojtylak, Global and local behavior of zeros of nonpositive type,  \textit{J.  Math. Anal.  Appl.}, 414 (2014) 273--284 }.



\bibitem{Stange1}{P. Stange, On the Efficient Update of the Singular Value Decomposition
Technischer Report, TU Braunschweig, Institut Computational Mathematics, 2011}

\bibitem{Stange2}{P. Stange, On the Efficient Update of the Singular Value Decomposition, {\it
 79th Annual Meeting of the International Association of Applied Mathematics and Mechanics (GAMM)}, vol. 8, Bremen, PAMM, 2008.}




\bibitem{Thompson}
R.C. Thompson,  Invariant factors under rank one perturbations, \emph{Canad. J. Math.} 32
(1980), 240-245.

\bibitem{Thompson76}
R.C. Thompson, The behavior of eigenvalues and singular values under perturbations of restricted rank, \emph{Linear Algenra and Applications}, 13 (1976), 69--78.


\bibitem{VL}
M.I. Vishik, L.A. Lyusternik, Solutions of some perturbation problems in the case of matrices
and self--adjoint and non--self--adjoint differential equations, \emph{Uspekhi Mat. Nauk }
(\emph{Russian Math. Surveys}) 15 (1960), 3-80.

\bibitem{VDN}
D.V. Voiculescu, K.J. Dykema, A. Nica, Free random variables. A noncommutative probability approach to free products with applications to random matrices, operator algebras and harmonic analysis on free groups. CRM Monograph Series, 1. American Mathematical Society, Providence, RI, 1992.

\bibitem
{wojakowski07}
{\L}. Wojakowski, Probability interpolating between free and Boolean. {\em Dissertationes Math.} 446 (2007), 45 pp.

\bibitem
{wojakowski10}
{\L}. Wojakowski, Two-levels $t$-transformation. {\em Banach Center Publ.} 80 (2010)
313--322.

\bibitem
{wysoczanski05}
J. Wysocza\'{n}ski,
\newblock Monotonic independence on the weakly monotone Fock space
and related Poisson type theorem.
\newblock {\em Infin. Dimens. Anal. Quantum Probab. Relat. Top.} 8 (2005), no.
2, 259--275.

\end{thebibliography}
\end{document}